\documentclass[a4, 12pt]{amsart}

\usepackage{amssymb,amsthm,amstext,amsmath,amscd,latexsym,amsfonts}
\usepackage{color}

\def\Hom{\mathrm{Hom}}
\def\RHom{\mathrm{{\bf R}Hom}}
\def\Ext{\mathrm{Ext}}

\def\Coker{\mathrm{Coker}}
\def\Ker{\mathrm{Ker}}
\def\Im{\mathrm{Im}}
\def\tr{\mathrm{tr}}
\def\Spec{\mathrm{Spec}}
\def\Supp{\mathrm{Supp}}
\def\Ann{\mathrm{Ann}}
\def\Ass{\mathrm{Ass}}
\def\Assh{\mathrm{Assh}}
\def\Min{\mathrm{Min}}
\def\sup{\mathrm{sup}}
\def\inf{\mathrm{inf}}

\def\ara{\mathrm{ara}}

\def\depth{\mathrm{depth}}
\def\dim{\mathrm{dim}}
\def\grade{\mathrm{grade}}
\def\height{\mathrm{ht}}
\def\pd{\mathrm{pd}}

\def\a{\mathfrak a}
\def\b{\mathfrak b}
\def\m{\mathfrak m}
\def\n{\mathfrak n}
\def\p{\mathfrak p}
\def\q{\mathfrak q}

\def\N{\mathbb N}

\def\P{\mathfrak P}
\def\Q{\mathfrak Q}
\def\L{\bf L}
\def\aa{\text{\boldmath $a$}}

\def\yy{\text{\boldmath $y$}}

\def\G{\varGamma }
\def\RG{\mathrm{{\bf R}{\Gamma}}}
\def\GI{\varGamma _{I}}
\def\GA{\varGamma _{\a}}
\def\GB{\varGamma _{\b}}
\def\GIJ{\varGamma _{I, J}}
\def\GmJ{\varGamma _{\m, J}}
\def\HI{H_{I}}
\def\Hm{H_{\m}}
\def\HA{H_{\a}}
\def\HIJ{H_{I, J}}
\def\HmJ{H_{\m , J}} 
\def\WIJ{W(I, J)}
\def\WmJ{W(\m, J)}
\def\WTIJ{\tilde{W} (I, J)}
\def\WTmI{\tilde{W} (\m, I)}
\def\WTmJ{\tilde{W} (\m, J)}

\theoremstyle{plain} 
\newtheorem{thm}{\textbf Theorem}[section]
\newtheorem{lem}[thm]{\textbf Lemma}
\newtheorem{cor}[thm]{\textbf Corollary}
\newtheorem{prop}[thm]{\textbf Proposition}

\theoremstyle{definition}
\newtheorem{df}[thm]{\textbf Definition}
\newtheorem{rem}[thm]{\textbf Remark}

\theoremstyle{plain}

\begin{document}
\title[Local cohomology based on a nonclosed support]{Local cohomology based on a nonclosed support defined by a pair of ideals}
\author{Ryo Takahashi, Yuji Yoshino, Takeshi Yoshizawa}
\email{}
\thanks{}

%
%
\begin{abstract}
We introduce a generalization of the notion of local cohomology module, 
which we call a local cohomology module with respect to a pair of ideals $(I, J)$, and study its various properties. 
Some vanishing and nonvanishing theorems are given for this generalized version of local cohomology. 
We also discuss its connection with ordinary local cohomology.
\end{abstract}
\maketitle
\section*{Introduction}

Local cohomology theory has been an indispensable and significant tool in commutative algebra and algebraic geometry.
In this paper, we intoroduce a generalization of the notion of local
cohomology module, which we call a local cohomology module with respect to a pair of ideals $(I, J)$, and study its various properties. 

To be more precise, let  $R$  be a commutative noetherian ring and let  $I$  and $J$  be ideals of  $R$.
We are concerned with the subset
$$
\WIJ = \{\,\p \in \Spec(R)\mid I^n \subseteq \p + J\text{ for an integer }n \geq 1\,\}
$$
of $\Spec(R)$.
See Definition \ref{def-WIJ} and Corollary \ref{KTY1-2} $(\rm 1)$.
In general, $\WIJ$ is closed under specialization, but not necessarily a closed subset of $\Spec (R)$.
For an $R$-module $M$, we consider the {\it $(I, J)$-torsion submodule} $\GIJ (M)$ of $M$ which consists of all elements $x$ of $M$ with $\Supp (Rx) \subseteq \WIJ$.
Furthermore, for an integer $i$, we define the $i$-th {\it local cohomology functor} $\HIJ^{i}$ with respect to $(I, J)$ to be  the $i$-th right derived functor of $\GIJ$. 
 We call  $\HIJ^{i}(M)$ the  $i$-th {\it local cohomology module} of $M$ with respect to $(I, J)$.
See  Definitions \ref{def-TF} and \ref{def-LCF}.

Note that if $J=0$ then $\HIJ ^{i}$ coincides with the ordinary local cohomology functor $\HI ^{i}$ with the support in the closed subset  $V(I)$.
On the other hand, if  $J$  contains $I$ then it is easy to see that  $\GIJ$  is the identity functor and  $\HIJ ^i = 0$ for  $i >0$. 
Thus we may consider the local cohomology functor  $\HIJ ^i$  as a family of functors with parameter $J$, which connects the ordinary local cohomology functor  $H_I^i$  with the trivial one. 

Our main motivation for this generalization is the following. 
Let  $(R, \m)$  be a local ring and let  $I$  be an ideal of  $R$. 
We assume that  $R$  is a complete local ring for simplicity. 
For a finitely generated $R$-module  $M$  of dimension $r$, P. Schenzel \cite{Schenzel} introduces the notion of the canonical module  $K_M$, and he proves the existence of a monomorphism  $H^r _I (M) ^{\vee} \to K_{M}$  and determines the image of this mapping, where  ${}^{\vee}$ denotes the Matlis dual. 
By his result, we can see that the image is actually equal to  $\Gamma _{\m, I} (K_{M})$. 
From this observation one expects that there would be a duality between the ordinary cohomology functor  $H_I^{i}$  and  our cohomology functor $H_{\m, I}^{i}$. 
We shall show in Section 5 that there are canonical isomorphisms 
$$
H^r _I (M) ^{\vee} \simeq \Gamma _{\m, I}(K_{M}) \quad \text{and} \quad 
H^r _{\m, I} (M) ^{\vee} \simeq \Gamma _{I}(K_{M}). 
$$
See Theorem \ref{c2} and Corollary \ref{c2cor}.

We should note that our idea already appears in several articles, but in more general setting. 
In fact, if we denote by  $\WTIJ$  the set of ideals  $\a$  satisfying  $I^n \subseteq \a +J$ for an integer  $n$, then the set $F=\{ D(\a) \mid \a \in \WTIJ \}$ of open sets in $\Spec(R)$ forms a Zariski filter on $\Spec (R)$. 
See \cite[Definition 6.1.1]{Brenner}.
In this setting, H. Brenner \cite[Section 6.2]{Brenner} defines the functor $\G_F$  by 
$$
\G _{F}(M)=\{ x \in \G (\Spec(R), M) \mid \, x|_{V(\a)}=0 \text{ for some } D(\a) \in F \} 
= \varinjlim _{D(\a) \in F} \G_{V(\a)}(M), 
$$
for an $R$-module  $M$. 
This actually coincides with $\GIJ (M)$.

The aim of the present paper is to generalize a number of
statements about ordinary local cohomology to our generalized local cohomology  $H_{I, J}^{i}$. 
One of our main goals is to give criteria for the vanishing and nonvanishing of  $H_{I,J}^i(M)$.


\vspace{6pt}

The organization of this paper is as follows.

After discussing basic properties of the local cohomology functors $\HIJ^{i}$ and the subset $\WIJ$ of $\Spec(R)$ in Section $1$, 
we define a generalization of  \v{C}ech complexes in Section $2$. 
In fact, we show that the local cohomology modules with respect to  $(I, J)$  are obtained as cohomology modules of the generalized \v{C}ech complexes  (Theorem \ref{L=C}). 

In Section $3$, we show some relationship of our local cohomology functor with the ordinary local cohomology functor. 

Section $4$ is a core part of this paper, where we discuss the vanishing and nonvanishing of $\HIJ ^i$.
We are interested in generalizing Grothendieck's vanishing theorem and Lichtenbaum-Hartshorne theorem to our context. 
In fact, one of our main theorems says that the equality 
\[ \inf \{~ i \mid \HIJ ^i (M) \neq 0 ~ \} = \inf \{~ \depth ~ M_{\p} \mid \p \in \WIJ ~ \}  \]
holds for a finitely generated module $M$ (Theorem \ref{KTY1-6}).
A generalized version of Lichtenbaum-Hartshorne theorem will be given in Theorem \ref{GLHVT}.

In Section $5$, we shall show a generalized version of the usual local duality theorem for local cohomology modules with respect to $(I, J)$. 
Also, motivated by the work of Schenzel, we discuss some kind of duality between  $\HIJ ^i$  and ordinary local cohomology modules. 
See Theorem \ref{LD}.

In Section $6$, we study the right derived functor  $\RG  _{I, J}$  defined on the derived category $D^b(R)$, and prove several functorial identities involving  $\RG _{I,J}$. 
See Theorems \ref{p-38} and \ref{p-53}.

\vspace{6pt}

Throughout the paper, we freely use the conventions of the notation for 
commutative algebra from the books Bruns-Herzog \cite{BH} and Matsumura \cite{Matsumura}.  
And we use well-known theorems concerning ordinary local cohomology without citing any references, for which the reader should consult Brodmann-Sharp \cite{BS}, Foxby \cite{Foxby}, Grothendieck \cite{Grothendieck} and Hartshorne \cite{Hartshorne}.

\section{Definition and Basic Properties}

Throughout this paper, we assume that all rings are
commutative noetherian rings. Let $R$ be a ring, and $I,J$ ideals of
$R$.

\begin{df}\label{def-TF}
For an  $R$-module $M$,  we denote by $\GIJ (M)$ the set of elements $x$ of $M$ such that $I^n x\subseteq J x$ for some integer $n$. 
$$
\GIJ (M ) = \{ x \in M \ | \ I^n x\subseteq J x \quad \text{for } \ \ n \gg 1\}
$$
Note that an element $x$ of $M$ belongs to $\GIJ (M)$ if and only if 
$I^n \subseteq  \Ann (x)+J $ for $n \gg 1$. 
Using this, we easily see that $\GIJ(M)$ is an $R$-submodule of $M$. 

For a homomorphism $f:M\to N$ of $R$-modules, it is easy to see that the inclusion $f(\GIJ (M)) \subseteq \GIJ (N)$, and hence the mapping $\GIJ (f) : \GIJ (M) \to \GIJ (N)$ is defined so that it agrees with $f$ on $\GIJ (M)$. 

Thus  $\GIJ $  becomes an additive covariant functor from the category of all $R$-modules to itself. 
We call $\GIJ$ the {\it $(I, J)$-torsion functor}. 
\end{df}

It is obvious that if $J=0$, then the $(I, J)$-torsion functor $\GIJ$ coincides with $I$-torsion functor $\GI$.

\begin{lem}
The $(I, J)$-torsion functor $\GIJ$ is a left exact functor on the category of all $R$-modules.
\end{lem}

\begin{proof}
Let $0 \to L \overset{f}{\to } M \overset{g}{\to } N \to 0$  be an exact sequence of $R$-modules. 
We must show that 
\[\minCDarrowwidth1.5pc \begin{CD}  0 @>>> \GIJ (L) @>\GIJ(f)>> \GIJ(M) @>\GIJ(g)>> \GIJ(N) \end{CD} \]
is exact. 
It is clear that $\GIJ(f)$ is a monomorphism and 
\[ \Im (\GIJ(f)) \subseteq \Ker (\GIJ(g)).\]  
To prove the converse inclusion, let $x \in \Ker (\GIJ (g))$. 
Since $x \in \GIJ (M)$, there exists an integer $n\geq 0$ such that $I^{n} x \subseteq Jx$. 
There is an element $y \in L$ with $f(y)=x$, since $g(x)=0$. 
We have to show that $y \in \GIJ(L)$. 
For each $a \in I^{n}$, we have $f(ay)=af(y)=ax \in I^{n}x \subseteq Jx$, 
and hence there is an element  $b\in J$ with  $ax=bx$. 
Thus the equality $f((a-b)y)=af(y)-bf(y)=ax-bx=0$ holds, and consequently 
$(a-b)y=0$ because $f$ is a monomorphism. 
Therefore $ay\in Jy$, and thus  $I^{n}y \subseteq Jy$. 
It follows that $y\in \GIJ(L)$.
\end{proof}

\begin{df}\label{def-LCF}
For an integer $i$, the $i$-th right derived functor of $\GIJ$ is denoted by $\HIJ ^{i}$  and will be referred  to as the {\it $i$-th local cohomology functor with respect to $(I, J)$}. 

For an $R$-module $M$, we shall refer to $\HIJ ^{i} (M)$ as the {\it $i$-th local cohomology module} of $M$ {\it with respect to $(I, J)$}, 
and to $\GIJ (M)$ as the {\it $(I, J)$-torsion part} of $M$. 

We say that $M$ is {\it $(I, J)$-torsion} (respectively {\it $(I, J)$-torsion-free}) precisely  when $\GIJ (M)=M$ (respectively $\GIJ (M)=0$).
\end{df}

It is easy to see that if  $J=0$, then  $\HIJ ^{i}$ coincides with the ordinary local cohomology functor $\HI ^{i}$.

We collect some basic properties of the $(I, J)$-torsion part and the local cohomology modules with respect to $(I, J)$.

\begin{prop}\label{replace IJ}
Let $I$, $I^{\prime}$, $J$, $J^{\prime}$ be ideals of $R$ and let $M$ be an $R$-module. 

\begin{itemize}
\item[{(\rm 1)}]\ $\G_{I, J}(\G_{I^{\prime}, J^{\prime}}(M))=  \G_{I^{\prime}, J^{\prime}}(\G_{I, J}(M))$. \vspace{4pt}

\item[{(\rm 2)}]\ If $I\subseteq I^{\prime}$, then $\GIJ(M) \supseteq  \G_{I^{\prime}, J}(M)$.\vspace{4pt}

\item[{(\rm 3)}]\ If $J\subseteq J^{\prime}$, then $\G_{I, J}(M) \subseteq  \G_{I, J^{\prime}}(M)$.\vspace{4pt}

\item[{(\rm 4)}]\ $\G _{I, J} (\G _{I^{\prime}, J} (M) )=\G _{I+I^{\prime}, J} (M)$.\vspace{4pt}

\item[{(\rm 5)}]\ $\G_{I, J}(\G_{I, J^{\prime}}(M))=\G_{I, JJ^{\prime}}(M)=\G_{I, J\cap J^{\prime}}(M)$.\\ 
In particular, $H_{I, JJ^{\prime}}^{i} (M)=H_{I, J\cap J^{\prime}}^{i}(M)$ for all integers $i$.\vspace{4pt}

\item[{(\rm 6)}]\ If $J^{\prime} \subseteq J$, then $H_{I+J^{\prime}, J}^{i}(M)=\HIJ^{i}(M)$ for all integers $i$. \\
In particular, $H_{I+J, J}^{i} (M)=\HIJ^{i}(M)$ for all integers $i$.\vspace{4pt}

\item[{(\rm 7)}]\ If $\sqrt{I}=\sqrt{I^{\prime}}$, then $\HIJ ^{i} (M)=H_{I^{\prime}, J}^{i} (M)$ for all integers $i$. \\
In particular,  $\HIJ ^{i} (M)=H_{\sqrt{I}, J}^{i} (M)$ for all integers $i$.\vspace{4pt}

\item[{(\rm 8)}]\ If $\sqrt{J}=\sqrt{J^{\prime}}$, then $\HIJ ^{i} (M)=H_{I, J^{\prime}}^{i} (M)$ for all integers $i$. \\
In particular,  $\HIJ ^{i} (M)=H_{I, \sqrt{J}}^{i} (M)$ for all integers $i$.\vspace{4pt}
\end{itemize}
\end{prop}

\begin{proof}
All these statements follow easily from the definitions. 
As an illustration we just will prove statement $(\rm 4)$.

Let $x \in \G _{I, J} (\G _{I^{\prime}, J} (M) )$. 
Then there exists integers $m, n \geq 0$ such that $I^{m}x \subseteq Jx$ and $I^{\prime n}x \subseteq Jx$ hold. 
Thus we have $(I+I^{\prime})^{m+n} x \subseteq I^{m}x+I^{\prime n}x \subseteq Jx$, and hence $x\in \G_{I+I^{\prime}, J}(M)$. 
To prove the converse inclusion, let $x \in \G _{I+I^{\prime}, J} (M)$. 
Then there exists an integer $n\geq 0$ such that $(I+I^{\prime})^{n}x \subseteq Jx$. 
Thus $I^{n}x, I^{\prime n} x \subseteq (I+I^{\prime})^{n}x \subseteq Jx$. 
Hence $x\in \G _{I, J} (\G _{I^{\prime}, J} (M) )$. 
\end{proof}

\begin{df}\label{def-WIJ}
Let $\WIJ $ denote the set of prime ideals $\p $ of $R$ such that $I^n \subseteq J+\p $ for some integer $n$.
$$
\WIJ = \{ \p \in \Spec (R) \ | \ I^n \subseteq J+\p \quad \text{for} \ \ n \gg 1\}
$$ 
\end{df}

It is easy to see that if $J=0$, then  $\WIJ $ coincides with the Zariski closed set $V(I)$ consisting of all prime ideals containing $I$.
Note that $\WIJ $ is stable under specialization, 
but in general, it is not a closed subset of  $\Spec (R)$.

We exhibit  some of the properties of $\WIJ$ below.

\begin{prop}
Let $I$, $I^{\prime}$, $J$, $J^{\prime}$ be ideals of $R$. 
\begin{itemize}

\item[{(\rm 1)}]\ If $I\subseteq I^{\prime}$, then $\WIJ \supseteq W(I^{\prime}, J)$. \vspace{4pt}

\item[{(\rm 2)}]\ If $J\subseteq J^{\prime}$, then $\WIJ \subseteq W(I, J^{\prime})$. \vspace{4pt}

\item[{(\rm 3)}]\ $W(I+I^{\prime}, J) = \WIJ \cap W(I^{\prime}, J)$. \vspace{4pt}

\item[{(\rm 4)}]\ $W(I, J J^{\prime})=W(I, J \cap J^{\prime})=W(I, J)\cap W(I, J^{\prime})$. \vspace{4pt}

\item[{(\rm 5)}]\ $\WIJ=W(\sqrt{I}, J)=W(I, \sqrt{J})$. \vspace{4pt}

\item[{(\rm 6)}]\ Let  $R$  be a local ring with maximal ideal  $\m$. 
If $I$ is not an $\m$-primary ideal, then the following equality holds. 
\[ W(\m, I)=(\bigcap_{I \subsetneq J} \WmJ )\cap \{ \p \mid \p \text{ is prime ideal such that } \p \not \subseteq I\}.\] 

\item[{(\rm 7)}]\ $V(I) = \bigcap_{J} \WIJ = \bigcap _{J\in D(I)} \WIJ$, 
where $D(I)$ is the complement of $V(I)$ in $\Spec (R)$. 
\end{itemize}
\end{prop}

\begin{proof}
$(\rm 1)$ to $(\rm 5)$:  The proofs are easy.
We will only prove the statement $(\rm 4)$ and leave the proofs of the remaining statements to the reader. 

Since $J J^{\prime} \subseteq J \cap J^{\prime} \subseteq J, \ J^{\prime}$, 
it holds that $W(I, J J^{\prime}) \subseteq  W(I, J \cap J^{\prime}) \subseteq W(I, J)\cap W(I, J^{\prime})$.
Let $\p \in W(I, J)\cap W(I, J^{\prime})$. 
Then there exists integers $m$, $n\geq 0$ such that $I^{m}\subseteq \p +J$, $I^{n}\subseteq \p +J^{\prime}$. 
Thus $I^{m+n} \subseteq (\p+J)(\p+J^{\prime})\subseteq \p+J J^{\prime}$. 
Hence we have $\p \in W(I, J J^{\prime})$. 

\noindent
$(\rm 6)$:  
Let $\p \in W(\m, I)$, then $I+\p$ is $\m$-primary. 
If $I\subsetneq J$, then $J+\p$ is $\m$-primary  as well, hence $\p \in \WmJ$. 
Since $I$ is not $\m$-primary, we have $\p \not \subseteq I$. 

To prove the converse, let  $\p \in \bigcap_{I \subsetneq J} \WmJ$ with $\p \not \subseteq I$. 
Setting $J = I+\p \supsetneq I$, we must have $\p \in W(\m, I+\p)$. 
Thus $I+\p$ is an $\m$-primary ideal. 
Therefore it follows  $\p \in W(\m, I)$.

\noindent
$(\rm 7)$: 
It is trivial that $V(I) \subseteq \bigcap_{J} \WIJ \subseteq \bigcap _{J\in D(I)} \WIJ$. 
Suppose that $\p \not \in V(I)$. Then we have $\p \in D(I)$ and $\p \not \in W(I, \p)$. 
Thus $\p \not \in \bigcap _{J\in D(I)} \WIJ$. 
\end{proof}

\begin{prop}\label{KTY1-1}
For an $R$-module $M$, the following are equivalent.
\begin{enumerate}
  \item[{(\rm 1)}]\ $M$ is $(I, J)$-torsion $R$-module. \vspace{4pt}
  \item[{(\rm 2)}]\ $\Min (M) \subseteq \WIJ $. \vspace{4pt}
  \item[{(\rm 3)}]\ $\Ass (M) \subseteq \WIJ $. \vspace{4pt}
  \item[{(\rm 4)}]\ $\Supp (M) \subseteq \WIJ $. \vspace{4pt}
\end{enumerate}
\end{prop}

\begin{proof}
The implications $(\rm 4) \Rightarrow (\rm 3) \Rightarrow (\rm 2)$ are trivial.

\noindent 
$(\rm 2) \Rightarrow (\rm 4)$ : 
For $\p \in \Supp (M)$, there exists $ \q \in \Min (M)$ such that $\q \subseteq \p$.
Since $\q \in \WIJ$,  $I^n \subseteq J+\q \subseteq J+\p $ for an integer $n$. Hence $\p \in \WIJ$. 

\noindent 
$(\rm 1) \Rightarrow (\rm 3)$ : 
If $\p \in \Ass (M)$ then $\p =\Ann(x)$ for some $x \in M$. 
Since $M$ is an $(I, J)$-torsion $R$-module, there exists an integer $n$ such that  $I^n \subseteq J+\Ann (x) = J+\p $. 
Hence $\p \in \WIJ$. 

\noindent 
$(\rm 4) \Rightarrow (\rm 1)$ : 
We have to show that $M \subseteq \GIJ (M)$.
Let $x \in M$, and set $\Min (Rx)=\{ \p_{1}, \ldots ,\p_{s} \}$. 
Since $\Min (Rx) \subseteq \Supp (M) \subseteq \WIJ$, there exists an integer $n$ such that $I^n \subseteq J+\p_{i}$ for all $i$,  thus  $I^{ns} \subseteq J+(\p_{1} \cdots \p_{s})$. 
Now since $\sqrt{\Ann (x)} = \p_{1} \cap \cdots \cap \p_{s} \supseteq 
\p_{1} \cdots \p_{s}$, it follows that $(\p_{1} \cdots \p_{s})^m \subseteq \Ann (x)$ for an integer $m$. 
Therefore we have $I^{mns} \subseteq J+\Ann (x)$. 
Hence $x \in \GIJ (M)$. 
\end{proof}

\begin{cor}\label{KTY1-2}

\par
\noindent 
\begin{enumerate}
\item[{(\rm 1)}] 
For $x \in M$, the following conditions are equivalent.
	\begin{enumerate}
	\item[{(\rm a)}]\ $x \in \GIJ (M)$. \vspace{4pt} 
	\item[{(\rm b)}]\ $\Supp (Rx) \subseteq \WIJ$. \vspace{4pt} 
	\end{enumerate} \vspace{4pt} 
\item[{(\rm 2)}]  
Let $0\to L\to M\to N\to 0$ be an exact sequence of $R$-modules.
Then $M$ is an $(I, J)$-torsion module if and only if $L$ and  $N$ are $(I, J)$-torsion modules.
\end{enumerate}
\end{cor}

\begin{proof}
$(\rm 1)$:
 $(\rm a) \Rightarrow (\rm b)$  
The assumption implies that $\GIJ (Rx)=Rx$. 
Thus by Proposition \ref{KTY1-1} we get $\Supp (Rx) \subseteq \WIJ$. 

\noindent 
$(\rm b) \Rightarrow (\rm a)$  
By using Proposition \ref{KTY1-1}, we get $x \in Rx = \GIJ (Rx) \subseteq \GIJ (M)$. 

\noindent 
$(\rm 2)$:  This follows from Proposition \ref{KTY1-1} and the fact that $\Supp (M)=\Supp(L) \cup \Supp(N)$.
\end{proof}

\begin{cor}
If $M$ is an  $(I, J)$-torsion $R$-module, then $M/JM$ is an  $I$-torsion $R$-module. 
The converse holds if $M$ is a finitely generated $R$-module.
\end{cor}

\begin{proof}
Since $M$ is an $(I, J)$-torsion $R$-module, we have $\Supp (M) \subseteq \WIJ$.  Thus we get  
$\Supp (M/JM) \subseteq \Supp (M) \cap V(J) \subseteq \WIJ \cap V(J) \subseteq V(I)$. 
Therefore $M/JM$ is $I$-torsion $R$-module. 

Suppose that $M$ is a finitely generated $R$-module, and let $x \in M$. 
We want to show that $x \in \GIJ(M)$. 
By the Artin-Rees lemma, there is an integer $n \geq 0$ such that $J^{n}M \cap Rx \subseteq Jx$. 
Since $M/JM$ is $I$-torsion, we have $\Supp (M/J^{n}M) =\Supp(M/JM)\subseteq V(I)$, therefore  $M/J^{n}M$ is $I$-torsion as well. 
Thus there exists an integer $m\geq 0$ with $I^{m}x \subseteq J^{n}M$. 
Hence it follows that $I^{m}x \subseteq J^{n}M \cap Rx \subseteq Jx$. 
Thus  $x \in \GIJ (M)$, as desired.
\end{proof}

\begin{prop}\label{KTY1-3}
Let  $M$  be an $R$-module. 
Then the equality 
$$
\Ass(M)\cap\WIJ = \Ass (\GIJ (M))
$$
holds. 
In particular, $\GIJ (M) \neq 0$ if and only if  $\Ass(M)\cap\WIJ \neq \emptyset$. 
\end{prop}

\begin{proof}
Since $\GIJ (M)$ is an $(I, J)$-torsion $R$-module, we have  $\Ass (\GIJ (M)) \subseteq \WIJ$ by Proposition \ref{KTY1-1}.
Thus the inclusion $\Ass(M) \cap \WIJ \supseteq \Ass (\GIJ (M))$ is obvious. 

To prove the converse inclusion, take  $\p \in \Ass(M)\cap\WIJ$. 
Then there is an element  $x (\not= 0) \in M$  with  $\p = \Ann (x)$ and an integer $n$  with  $I^n \subseteq J+\p$.
Thus  $I ^n \subseteq J + \Ann (x)$, hence  $x \in \GIJ (M)$.
Since  $\p = \Ann (x)$, we have  $\p \in \Ass (\GIJ (M))$. 
\end{proof}

For a prime ideal  $\p \in \Spec (R)$, we denote by  $E(R/\p)$  the injective hull of the $R$-module $R/\p $. 

\begin{prop}\label{KTY1-4}
Let  $\p \in \Spec (R)$. 
If $\p \in \WIJ$, then  $E(R/\p )$ is an $(I, J)$-torsion $R$-module. 
On the other hand, if $\p \not \in \WIJ$ then $E(R/\p )$ is an $(I, J)$-torsion-free $R$-module. 
\end{prop}

\begin{proof}
If  $\p \in \WIJ$, then $\Ass (E(R/\p ))=\{ \p \} \subseteq \WIJ$. 
Therefore $\GIJ (E(R/\p ))=E(R/\p )$ by Proposition \ref{KTY1-1}. 
Contrarily, if $\p \not \in \WIJ$, then $\Ass (E(R/\p )) \cap \WIJ = \{ \p \} \cap \WIJ =\emptyset $. 
Therefore, by Proposition \ref{KTY1-3}, we have $\GIJ (E(R/\p ))=0$.
\end{proof}

\begin{prop}\label{KTY1-5-1}
Let $M$ be an $(I, J)$-torsion $R$-module. 
Then there exists an injective resolution of $M$ in which
each term is an $(I, J)$-torsion $R$-module.
\end{prop}

\begin{proof}
First note that the injective hull $E^{0}$ of $M$ is also an $(I, J)$-torsion module. 
In fact, since $M$ is $(I, J$)-torsion, we have  $\Ass (E^{0})=\Ass (M) \subseteq \WIJ$ by Proposition \ref{KTY1-1}.  
Hence  $E^{0}$ is $(I, J$)-torsion. 
Thus we see that $M$ can be embedded in an $(I, J)$-torsion injective $R$-module $E^{0}$.

Suppose, inductively, we have constructed an exact sequence
\[ \minCDarrowwidth2pc \begin{CD} 0 @>>> M @>>> E^{0} @>>> \cdots @>>> E^{n-1} @>{d^{n-1}}>> E^{n} \end{CD} \]
of $R$-modules in which $E^{0},\ldots ,E^{n-1},E^{n}$ are $(I, J)$-torsion injective $R$-modules. 
Let $C$  be the cokernel of the map $d^{n-1}$. 
Since $E^{n}$ is an $(I, J)$-torsion module, $C$ is $(I, J)$-torsion as well by  Corollary \ref{KTY1-2} $(\rm 2)$. 
Applying the argument in  the first paragraph to $C$, we can embed $C$ into an $(I, J)$-torsion injective $R$-module $E^{n+1}$. 
This completes the proof by induction. 
\end{proof}

\begin{cor}\label{KTY1-5}
Let  $M$  be an $R$-module.

\par
\noindent
\begin{enumerate}
\item[{(\rm 1)}]
If $M$ is an $(I, J)$-torsion $R$-module, then $\HIJ ^{i} (M)=0$ for all $i>0$. 
 \vspace{4pt} 

\item[{(\rm 2)}]
$\HIJ ^{i} (\GIJ (M))=0$ for $i>0$.  \vspace{4pt} 

\item[{(\rm 3)}]
$M/\GIJ (M)$ is an $(I, J)$-torsion-free $R$-module. \vspace{4pt} 

\item[{(\rm 4)}]
There is an isomorphism  $\HIJ ^{i} (M) \cong \HIJ ^{i} (M/\GIJ (M))$ for all $i>0$. \vspace{4pt} 

\item[{(\rm 5)}]
$\HIJ ^i (M)$ is an $(I, J)$-torsion $R$-module for any integer $i\geq 0$.
\end{enumerate}
\end{cor}

\begin{proof}
(\rm 1) follows from Proposition \ref{KTY1-5-1}. 
Since $\GIJ (M)$ is an $(I, J)$-torsion $R$-module, (2) follows from (1). 

From the obvious exact sequence 
\[ 0 \to \GIJ (M) \to M \to M/\GIJ (M) \to 0 \]
we have an exact sequence  
$$
0 \to \GIJ(\GIJ (M)) \to \GIJ (M) \to \GIJ (M/\GIJ (M)) \to 0
$$
and isomorphisms 
$$ 
\HIJ ^i (M) \ \  {\cong} \ \ \HIJ ^i(M/\GIJ (M)) \qquad \text{for $i\geq 1$},  
$$
since  $\HIJ ^i (\GIJ(M)) =0$  for $i>0$. 
It follows from this that (3) and (4) hold. 

Since $\HIJ ^{i} (M) \ (i\geq 0)$ is a subquotient of an $(I, J)$-torsion module, it is also $(I, J)$-torsion by Corollary \ref{KTY1-2}, hence (5) holds. 
\end{proof}

\begin{rem}
In Corollary \ref{KTY1-5} $(\rm 1)$, the converse holds if $R$ is a local ring and $M$ is a finitely generated $R$-module. 
Namely, if $\HIJ ^{i} (M)=0$ for all integer $i>0$, then   $M$ is an $(I, J)$-torsion $R$-module. 
(See Corollary \ref{KTY1-7} below. )
\end{rem}


\section{ \v{C}ech Complexes}

In this section we present a generalization of \v{C}ech complexes. The main purpose is to show that the local cohomology modules with respect to  $(I, J)$  are obtained as the homologies of the generalized \v{C}ech complexes.

As before,  $I$, $J$  denote ideals of  a commutative noetherian ring $R$.

\begin{df}
For an element  $a \in R$,  let $S_{a,J}$ be the subset of  $R$  consisting of all elements of the form  $a^{n} + j$  where  $n \in \N$ and  $j\in J$. 
$$
S_{a, J} = \{ a ^n + j \ | \ n \in \N, \ j \in J \}
$$
Note that  $S_{a, J}$ is a multiplicatively closed subset of  $R$. 
For an $R$-module $M$, we denote by $M_{a, J}$ the module of fractions of $M$ with respect to $S_{a, J}$. 
$$
M_{a, J} = S_{a, J}^{-1} M 
$$
\end{df}

\begin{df}\label{GC}
For an element  $a \in R$, the complex $C^{\bullet }_{a, J}$ is defined as  
\[ 
C^{\bullet }_{a, J} = ( 0 \to R \to R_{a, J} \to 0 ), 
\]
where  $R$ is sitting in the $0$th position and $R_{a, J}$ in the $1$st position in the complex. 
For a sequence $\aa =a_{1}, \dots , a_{s}$ of elements of $R$, 
we define a complex $C^{\bullet }_{\aa, J}$ as follows:
\begin{align*}
 C^{\bullet }_{\aa, J}    
&= \bigotimes ^{s}_{i=1} C^{\bullet }_{a_{i}, J}\\
&= \left( 0 \to R \to \prod ^{s}_{i=1} R_{a_{i}, J} 
		       \to \prod _{i<j} (R_{a_{i}, J})_{a_{j}, J}  
\to \cdots \to ( \cdots (R_{a_{1}, J}) \cdots ) _{a_{s}, J} \to 0  \right)
\end{align*}
\end{df}

It is easy to see that if  $J=0$, then   $C^{\bullet }_{\aa, J}$ coincides with  the ordinary \v{C}ech complex $C^{\bullet }_{\aa}$ with respect to $\aa=a_{1}, \ldots , a_{s}$.

The following result gives some basic properties of the generalized \v{C}ech complexes.

\begin{prop}\label{basic Cech}
Let  $a \in R$. 

\par
\noindent
\begin{enumerate}
\item[{(\rm 1)}]
$S_{a, J}$ contains  $0$  if and only if $a \in \sqrt{J}$. 
\vspace{4pt}
\item[{(\rm 2)}]
If $a \in \sqrt{J}$, then $C^{\bullet }_{a, J} \cong  R$  as chain complexes. 
\vspace{4pt}
\item[{(\rm 3)}]
A prime ideal $\p$ belongs to $\WIJ$ if and only if $\p \cap S_{a, J}\neq \emptyset$ for any  $a \in I$. 
\vspace{4pt}
\item[{(\rm 4)}]
If $a \in I$, then $\HIJ ^{i}(M_{a, J})=0$ for all $i\geq 0$. 
\vspace{4pt}
\item[{(\rm 5)}]
If  $\sqrt{I}=\sqrt{(a_{1}, a_{2}, \ldots ,a_{s})}$, then the sequence 
\[ 0 \to \GIJ (M) \to M \to \prod^{s}_{i=1} M_{a_{i}, J} \]
is exact. 
\end{enumerate}
\end{prop}

\begin{proof}

$(\rm 1)$  
If $0 \in S_{a, J}$, then  $0= a^{n}+j$  for an integer $n$ and $j \in J$. 
Then, since  $a^{n}=-j \in J$, we have $a\in \sqrt{J}$. 
Conversely, if $a \in \sqrt{J}$, then there is an integer $n\geq 0$ such that $a^{n}=j$  belongs to $J$. 
Thus $0=a^{n}+(-j) \in S_{a, J}$. 

\noindent
$(\rm 2)$ 
Suppose  $a \in \sqrt{J}$. 
It then follows from  $(\rm 1)$  that  $0 \in S_{a, J}$. 
Thus $R_{a,J}=0$, hence  $C^{\bullet }_{a, J}=(0\to R \to 0)$ from the definition. 

\noindent
$(\rm 3)$  
Assume  $\p \in \WIJ$  and take an element  $a \in I$. 
Then  $I^{n}\subseteq J+\p$  for an integer $n\geq 0$. 
Since $a^{n} \in I^{n}\subseteq J+\p$, there exist $j \in J$ and $c \in \p$ such that $a^{n}=j+c$. 
Thus we have $c=a^{n}+(-j)\in \p \cap S_{a, J}$. 

Conversely, assume  $\p \cap S_{a, J}\neq \emptyset$ for any $a \in I$.
Corresponding to each $a \in I$, we find an element  $c(a) \in \p \cap S_{a, J}$, which is of the form  $c(a) = a^{n(a)} + j(a)$  for an integer $n(a)$ and  $j(a) \in J$. 
Thus $a^{n(a)}=-j(a)+c(a) \in J+\p$. 
Since this is true for any $a \in I$, and since $I$ is finitely generated, we see  $I ^n \subseteq J + \p$ for some $n$,  hence $\p \in W(I, J)$.

\noindent
$(\rm 4)$  
Let $E^{\bullet }$ be an injective resolution of an $R$-module  $M$. 
Then $(E^{\bullet })_{a, J}$ is an $R$-injective resolution of  $M_{a, J}$. 
Hence  $\HIJ ^{i}(M_{a, J})=H^{i} (\GIJ ((E^{\bullet })_{a, J}))$. 
Describing  each $E^i$  as a direct sum of indecomposable injective modules  
$E^{i}=\bigoplus  _{\p \in \Spec (R)} E_{R}(R/\p)^{\mu_{i} (\p, M)}$, we have
\[ (E^{i})_{a, J}=\bigoplus _{\p \in \Spec (R)} E_{R}(R/\p)_{a, J} ^{\mu_{i} (\p, M) }
=\bigoplus _{\p \in \Spec (R)} E_{R_{a, J}}(R_{a, J}/\p R_{a, J})^{\mu_{i} (\p, M)}. \]
Therefore the following equality follows from  $(\rm 3)$ and the assumption  $a \in I$.  
\[ \GIJ ((E^{i})_{a, J})=\bigoplus _{\p \in W(I, J)} E_{R_{a, J}}(R_{a, J}/\p R_{a, J}) ^{\mu _{i}(\p,M)}=0 \]
It follows  $\HIJ ^{i}(M_{a, J})=0$.

\noindent
$(\rm 5)$ 
 It is enough to show that 
$x\in \G _{I, J} (M)$ if and only if $x \in \Ker (M\to \Pi ^{s}_{i=1} M_{a_{i}, J})$. 
Let $x\in \G _{I, J} (M)$. Then there exists an integer $n\geq 0$ such that 
$a_{i}^{n} x \in Jx$ for all $a_{i}$. Therefore, since $(a_{i}^{n}-b_{i})x=0$ for some $b_{i} \in J$ and  
$a_{i}^{n}-b_{i} \in S_{a_{i},J}$, we have $x \in \Ker (M\to \Pi ^{s}_{i=1} M_{a_{i}, J})$. 
Conversely, if $x \in \Ker(M\to \Pi ^{s}_{i=1} M_{a_{i}, J})$, 
then for each $i$ there exist an integer $n_{i}\geq 0$ and $b_{i} \in J$ such that $(a_{i}^{n_{i}}-b_{i})x=0$. 
Thus  $a_i^{n_i} x \in J x$ for each $i$. 
This shows that $I^{n}x \subseteq  Jx$ for a large integer $n$. 
Thus we have $x\in \GIJ (M)$.  
\end{proof}

\begin{thm}\label{L=C}
Let $M$ be an $R$-module, and let  $\aa =a_{1}, \dots , a_{s}$ be a sequence of elements of $R$  which generate $I$. 
Then there is a natural isomorphism $\HIJ ^{i} (M) \cong H^{i} (C^{\bullet }_{\aa, J}\otimes _{R} M)$   for any integer $i$.
\end{thm}

\begin{proof}
Note from Proposition \ref{basic Cech} $(\rm 5)$  that there is a functorial isomorphism 
\[H^{0}(C^{\bullet }_{\aa, J} \otimes M) \cong \GIJ (M).\]
Since  $\{ H^{i} (C^{\bullet }_{\aa, J}\otimes _{R} - )\ | \ i \geq 0 \}$  is  a cohomological sequence of functors, 
to prove the theorem we only have to show that $H^{i} (C^{\bullet }_{\aa, J} \otimes E)=0$  for any  $i>0$ and any injective $R$-module $E$.
To prove this, we may assume that $E=E_{R}(R/\p)$ where $\p$ is a prime ideal of $R$.
We proceed by induction on the length $s$ of the sequence $\aa$.

If $s=1$, then 
$$
C^{\bullet }_{\aa, J} \otimes E =  \left( 0 \to  E_R(R/\p) \to E_R(R/\p)_{a_1, J} \to  0 \right),  
$$
where  $E_R(R/\p)_{a_1, J}$  is isomorphic to  $E_R(R/\p)$  if $\p \not\in W((a_1), J)$,  and is $(0)$  if  $\p \in W((a_1), J)$. 
See Propositions \ref{KTY1-4} or \ref{basic Cech}. 
In either case, we have  $H^1 (C^{\bullet }_{\aa, J} \otimes E ) =0$.

Next we assume $s>1$, and set  $\aa ^{\prime} =a_{2}, \ldots , a_{s}$. 
Then we have the equality  $C^{\bullet }_{\aa, J} =C^{\bullet }_{a_{1}, J} \otimes _{R} C^{\bullet }_{\aa ^{\prime}, J}$.
Therefore there is a spectral sequence
$$
E^{p, q}_2 = H^{p}(C^{\bullet }_{a_{1}, J} \otimes H^{q}( C^{\bullet }_{\aa ^{\prime}, J} \otimes E(R/\p))) 
\Rightarrow H^{p+q} (C^{\bullet }_{\aa, J} \otimes E(R/\p)).
$$ 
Since $H^{q} (C^{\bullet }_{\aa ^{\prime}, J} \otimes E(R/\p))=0 $ for $q>0$  by the induction hypothesis, 
the spectral sequence degenerates, and we have isomorphisms 
\begin{align*}
H^{n} (C^{\bullet }_{\aa, J} \otimes E(R/\p))  
		&=H^{n}(C^{\bullet }_{a_{1}, J} \otimes H^{0} (C^{\bullet }_{\aa ^{\prime}, J} \otimes E(R/\p))) \\
		&=H^{n}(C^{\bullet }_{a_{1}, J} \otimes \G _{(\aa ^{\prime}), J} (E(R/\p))) \\
		&=H^{n}(0\to \G _{(\aa ^{\prime}), J} (E(R/\p))\to (\G _{(\aa ^{\prime}), J} (E(R/\p)))_{a_{1}, J}\to 0). 
\end{align*}
This shows that $H^{n} (C^{\bullet }_{\aa, J} \otimes E(R/\p))=0$ for  $n \geq 2$. 
Note from Proposition \ref{KTY1-4} that  
$\G _{(\aa ^{\prime}),  J} (E(R/\p))$  is either $E(R/\p)$  or $(0)$. 
Therefore it remains to show that $H^{1} (C^{\bullet }_{a_1, J} \otimes E(R/\p))=0$. 
But this is already done in the case $s =1$. 
\end{proof}

\begin{cor}\label{J-tor}
Let $\aa =a_{1}, \dots , a_{s}$ be a sequence of elements of $R$, set $I=(\aa )$ and let  $M$ be a $J$-torsion $R$-module. 
Then there is a natural isomorphism $C^{\bullet}_{\aa, J} \otimes_{R} M \cong C^{\bullet}_{\aa} \otimes_{R} M$. 
Hence $\HIJ ^{i} (M) \cong \HI^{i}(M)$ for any integer $i$.
\end{cor}

\begin{proof}
For an element  $a \in I$, there is a natural mapping  $\varphi : M_a \to M_{a, J}$  defined by   $\varphi (z/a^n) = z/a^n$. 
First we show that  $\varphi$ is an isomorphism. 

Suppose that $\varphi ( z/a^{n}) =0 \in M_{a,J}$. 
Then $(a^{m} - b) z=0$  for an integer  $m\geq 0$ and an element $b \in J$. 
Since  $a^m - b$  divides  $(a^{2^{\ell}m}-b^{2^{\ell}})$, we see 
$(a^{2^{\ell}m}-b^{2^{\ell}})z=0$ for all integers $\ell \geq 0$. 
Since $M$ is $J$-torsion, we have $b^{2^{\ell}}z=0$ for a large $\ell$. 
Thus $a^{2^{\ell}m}z = 0$, and we have $z/a^{n}=0 \in M_{a}$, which shows that $\varphi$ is injective.

Let $w=z/(a^{n} - b) \in M_{a, J}$ where $z \in M$  and  $b \in J$. 
Since $M$ is $J$-torsion, there exists an integer $\ell$ such that $b^{2^{\ell}}z=0$. 
Let us write   $a^{2^{\ell}n}-b^{2^{\ell}}= c(a^n - b)$ for an element  $c \in R$. Then we see $a^{2^{\ell}n} z = c(a^n - b)z$  in  $M$. 
Therefore $w=z/(a^{n}-b)=cz/a^{2^{\ell}n} \in M_{a, J}$. 
This shows that $\varphi$ is surjective.

We have shown that  $M_a \cong M_{a,J}$ for any $a \in I$. 
Thus we have  $C^{\bullet}_{a, J} \otimes M \cong C^{\bullet}_{a} \otimes M$ for any $a \in I$. 
Finally we have the isomorphisms of chain complexes: 
\begin{align*}
 C^{\bullet}_{\aa, J} \otimes M 
 &= C^{\bullet}_{a_1, J} \otimes C^{\bullet}_{a_2, J} \otimes \cdots \otimes C^{\bullet}_{a_{s}, J} \otimes M \\
 &\cong C^{\bullet}_{a_1} \otimes C^{\bullet}_{a_2} \otimes \cdots \otimes C^{\bullet}_{a_{s}} \otimes M \\
 &= C^{\bullet}_{\aa} \otimes M
\end{align*}
\end{proof}

We can show the follows from this by Theorem \ref{L=C}.

\begin{prop}\label{comm injlim}
The functors  $\HIJ ^i \ (i\geq 0)$  commute with inductive limits, i.e. if  $\{ M_{\lambda} \ | \ \lambda \in \Lambda \}$  is an inductive system, then there is a natural isomorphism 
$$
\HIJ ^i (\varinjlim _{\lambda} M_{\lambda} ) \cong \varinjlim _{\lambda} \HIJ ^i(M_{\lambda}),   
$$
for any $i \geq 0$.
\end{prop}

\begin{proof}
Since the tensor product commutes with direct limits, we have 
$C^{\bullet}_{\aa, J} \otimes _R (\varinjlim _{\lambda} M_{\lambda}) \cong 
\varinjlim _{\lambda} (C^{\bullet}_{\aa, J} \otimes _R  M_{\lambda} )$. 
The proposition follows from this. 
\end{proof}

The following theorem is a generalization of the base ring
independence theorem for ordinary local cohomology.

\begin{thm}\label{IDT}
Let  $I$  and  $J$  be ideals of  $R$ as before. 
Furthermore, let $\varphi : R \to R^{\prime} $ be a ring homomorphism,  
and let  $M^{\prime}$ be an $R^{\prime}$-module.
Suppose that $\varphi$ satisfies the equality 
$$
\varphi (J) = JR'. 
$$
Then there is a natural isomorphism
$\HIJ ^{i} (M^{\prime} ) \cong H^{i}_{IR^{\prime}, JR^{\prime}} (M^{\prime})$ 
as $R'$-modules for any integer $i\geq 0$.
\end{thm}

\begin{proof}
Set $I=(\aa) = (a_1, \ldots , a_s)R$ and $\varphi (\aa) = \varphi (a_1) ,\ldots, \varphi (a_s)$. 

Then we have from the assumption the equality 
$$
\varphi (S_{a_i, J}) = S_{\varphi (a_i), JR'}, 
$$
for any multiplicative closed subset $S$ in $R'$ and for all $i$ with $1 \leq i \leq s$. 
Therefore,  
$\HIJ ^{i} (M') \cong 
H^{i} (C^{\bullet }_{\aa, J} \otimes _{R} M') \cong 
H^{i}(C^{\bullet }_{\varphi(\aa), JR'} \otimes _{R'} M') \cong 
H^{i}_{IR', JR'} (M')$.
\end{proof}

Here we should remark that the hypothesis $\varphi(J)=JR'$
in the theorem cannot be deleted.  Indeed, let $k$ be a field, 
$R=k[x,y]$ and $R'=k[x,y,z]/(xz-yz^2)$.
Set $I=(x)R'$, $J=(y)R'$.
For a natural ring homomorphism $\varphi$ from $R$ to $R'$,
we have $\varphi(J) \subsetneq  JR'$ and $\Gamma _{I,J}(R')\neq \Gamma_{IR',JR'}(R')$.

If  $\varphi : R \to R'$ is a surjective ring homomorphism, then it satisfies the condition  $\varphi (J) = JR'$ of the theorem. 
However, note that there is a non-surjective ring homomorphism that satisfies the condition. 
For example, let  $R = k [x]$ be a polynomial ring over a field  $k$  and let   $R' = k[x, y]/(xy)$. 
We define a $k$-algebra map  $\varphi : R \to R'$  by  $\varphi (x) = x$. 
Then we have  $\varphi  ( xR) = xR'$.

\begin{rem}\label{FBT}
Let $\varphi : R \to R^{\prime} $ be a flat homomorphism of rings, and let  $M$ be an $R$-module.
Then it induces a natural mapping 
$\HIJ ^{i}(M) \otimes _{R} R^{\prime} \to  H^{i}_{IR^{\prime}, JR^{\prime}} (M\otimes _{R} R^{\prime})$  for any  $i \geq 0$.

In fact, since $\varphi (S_{a_i, J}) \subseteq S_{\varphi (a_i), JR'}$, 
we have a chain homomorphism 
$(C^{\bullet }_{\aa, J} \otimes _{R} M) \otimes _R R' \to C^{\bullet }_{\varphi(\aa), JR'}\otimes _{R'} (M \otimes _R R')$, 
which induces the mapping of cohomologies.

We should note that this induced mapping may not be an isomorphism.

In fact, one can easily construct an example of a localization map $R \to S^{-1}R$  such that  $S^{-1} \Gamma _{I, J} (R) \to \Gamma _{S^{-1}I, S^{-1}J}(S^{-1}R)$  is not surjective.

For a further nontrivial example, let  $\varphi : R = k[x, y]_{(x, y)}  \to \widehat{R} = k[[x, y]]$  be the completion map, and let  $I =xR$  and  $J=yR$.  Furthermore, let  $S =\{ x^n + ya\ | \ a \in R\}$  and  $\widehat{S} = \{ x^n + yb \ | \ b \in \widehat{R}\}$  be multiplicatively closed subsets in  $R$  and  $\widehat{R}$  respectively. 
Then we obtain through the computation using Theorem \ref{L=C} the following equalities.  
$$
\HIJ ^1(R) \otimes _R \widehat{R} = S^{-1}\widehat{R}/\widehat{R}, \qquad 
H_{I\widehat{R},J\widehat{R}} ^1 (\widehat{R}) = \widehat{S}^{-1}\widehat{R}/\widehat{R}. 
$$
It is easy to see that the natural mapping  $S^{-1}\widehat{R}/\widehat{R} \to  \widehat{S}^{-1}\widehat{R}/\widehat{R}$  is injective, but not surjective.

\end{rem}

\section{Relations Between $\HI^{i}$ and $\HIJ^{i}$}

In this section, we study the relations between the local cohomology functors $\HI^{i}$ and $\HIJ^{i}$. 
We need Theorem \ref{dir-lim} below in the proof of one of the vanishing theorems of local cohomologies. 
(See Theorem \ref{non-local van} $(\rm i)$  below.)
First  we introduce a necessary notation.

\begin{df}
Let $\WTIJ $ denote the set of ideals $\a $ of $R$ such that $I^n \subseteq \a +J$ for some integer $n$. 
We define a partial order on $\WTIJ$ by letting $\a \leq \b$ if $\a \supseteq \b$ for $\a$, $\b\in \WTIJ$.
If $\a \leq \b$, we have $\GA (M) \subseteq \GB (M)$. 
The order relation on  $\WTIJ$  and the inclusion maps make $\{ \GA (M) \}_{\a \in \WTIJ}$ into a direct system of $R$-modules. 
\end{df}

\begin{thm}\label{dir-lim}
Let $M$ be an $R$-module. 
Then there is a natural isomorphism 
$$ 
\HIJ ^{i} (M) \cong \varinjlim_{\a \in \WTIJ} \HA ^{i} (M)
$$ 
for any integer $i$.
\end{thm}

\begin{proof}
First of all, we  show that $\GIJ (M)=\bigcup _{\a\in\WTIJ} \GA (M)$.

To do this, suppose  $x \in \GIJ (M)$. 
Then there is an integer $n\geq0$ with $I^{n}\subseteq \Ann (x)+J$.
Setting  $\a=\Ann (x)$,  we have  $\a\in \tilde{W}(I, J)$, and $x \in \GA (M)$.
Conversely, let $x \in \bigcup _{\a\in \tilde{W} (I, J)} \GA (M)$. 
Then there is an ideal  $\a \in \tilde{W} (I, J) $  with  $x\in \GA (M)$. 
Thus  $I^{m} \subseteq \a +J$ and $\a ^{n}x=0$  for integers $m,~n \geq 0$.
Then, since $I^{mn}\subseteq  (\a +J)^n \subseteq \a ^{n}+J$, we have $I^{mn}x\subseteq Jx$, hence  $ x \in \GIJ (M)$.

Let $0 \to L \to M \to N \to 0$ be an exact sequence of $R$-modules. 
Then it implies a long exact sequence
\[ \begin{CD}
0 @>>> H^{0}_{\a} (L) @>>> H^{0}_{\a} (M) @>>> H^{0}_{\a} (N) \\
  @>>> H^{1}_{\a} (L) @>>> H^{1}_{\a} (M) @>>> \cdots 
\end{CD}\]
for each $\a \in \WTIJ$.
Since taking the direct limit is an exact functor, we obtain the long exact sequence
\[ \begin{CD}
0 @>>> \varinjlim_{\a \in \WTIJ} H^{0}_{\a} (L) 
  @>>> \varinjlim_{\a \in \WTIJ} H^{0}_{\a} (M) 
  @>>> \varinjlim_{\a \in \WTIJ} H^{0}_{\a} (N) \\
  @>>> \varinjlim_{\a \in \WTIJ} H^{1}_{\a} (L) 
  @>>> \varinjlim_{\a \in \WTIJ} H^{1}_{\a} (M) 
  @>>> \cdots .
\end{CD}\]
On the other hand, for any injective $R$-module $E$ and any positive integer $i$, we have $H^{i}_{\a}(E)=0$ for each $\a \in \WTIJ$. 
Thus we have $\varinjlim_{\a \in \WTIJ} H^{i}_{\a} (E)=0$. 

These arguments imply that  $\{ \varinjlim_{\a \in \WTIJ} H^i _{\a} \ | \ i = 0, 1, 2, \ldots \}$  is a system of right derived functors of  $\GIJ$, and the proof is completed.
\end{proof}

Next we shall show that in a local ring $R$ with maximal
ideal $\m$ the $I$-torsion functor $\GI$ has a description as an inverse limit of $(\m,J)$-torsion functors $\GmJ$.
The following lemma is a key for this fact.

\begin{lem}\label{inv-lem}
Let $R$ be a local ring with maximal ideal $\m$. 
Then 
\[ V(J)=\bigcap _{I\in \WTmJ} W(\m, I)=\bigcap _{\p \in \WmJ} W(\m, \p). \]
\end{lem}

\begin{proof}
If $\p \in V(J)$ and $I\in \WTmJ$, then 
$\m ^n \subseteq  I + J \subseteq  I + \p$  for an integer  $n > 0$, hence we have $\p \in W(\m, I)$. 
Thus  $V(J) \subseteq \bigcap _{I\in \WTmJ} W(\m, I)$. 
Since $\WmJ \subseteq \WTmJ$, we have $\bigcap _{I\in \WTmJ} W(\m, I) \subseteq \bigcap _{\p \in \WmJ} W(\m, \p)$.

We only have to show the remaining inclusion  $\bigcap _{\p \in \WmJ} W(\m, \p) \subseteq V(J)$. 
Suppose that $\bigcap _{\p \in \WmJ} W(\m, \p) \nsubseteq V(J)$. 
Then there is a prime ideal $\q \in \bigcap _{\p \in \WmJ} W(\m, \p)$ with 
 $\q \not \in V(J)$. 
Take an element $x \in J \setminus \q$ and set  $r = \dim\, R/\q$. 
Since $x$ is $R/\q$-regular element, $\dim\, R/(\q +(x))=r-1$. 
Thus there exist $y_{1}, y_{2}, \ldots ,y_{r-1} \in \m$ such that 
$\bar{y}_{1}, \bar{y}_{2} ,\ldots \bar{y}_{r-1} \in \m/(\q+(x))$ is a system of parameters of $R/(\q+(x))$. 
Then $\q +(x, y_{1}, y_{2}, \ldots, y_{r-1})$ is an $\m$-primary ideal, 
and $\q +(y_{1},y_{2}, \ldots, y_{r-1})$ is not. 
Thus we can find a prime ideal $\p$ with $\q +(y_{1},y_{2}, \ldots, y_{r-1}) \subseteq \p \subsetneq \m$. 
On the other hand,  $J+\p$ is an  $\m$-primary ideal, since  $\q +(x, y_{1},y_{2}, \ldots, y_{r-1}) \subseteq (x)+\p \subseteq J+\p$. 
Therefore $\p \in \WmJ$, and hence we must have   $\q \in W(\m, \p)$. 
Thus we conclude that $\p=\p+\q$ is an $\m$-primary ideal, but this is a contradiction. 
\end{proof}

Recall that $\WTIJ $ is a partially ordered set, 
in which the order relation  $\a \leq \b$ for $\a, \b \in \WTIJ$  is defined by  $\b \subseteq \a$. 
Note that the relation  $\a \leq \b$ naturally implies the inclusion mapping  $\G_{I, \a} (M) \supseteq  \G_{I, \b } (M)$, which makes $\{ \G_{I, \a} (M) \}_{\a \in \WTIJ}$ an inverse system of $R$-modules. 
We are now ready to prove the following proposition.

\begin{prop}\label{inv-lim}
Let $R$ be a local ring with maximal ideal $\m$, and $M$ be an $R$-module. 
Then we have the equality 
$$
\GI (M)= \varprojlim _{J \in \WTmI} \GmJ (M).
$$
\end{prop}

\begin{proof}
We show  $\GI (M)= \bigcap_{J \in \WTmI} \GmJ (M)$. 
For this, let $x\in \GI(M)$ and $J \in \WTmI$. 
Then there are integers $m$, $n \geq0$ with $I^{m}x=0$, $\m^{n}\subseteq J+I$. 
Thus  $\m^{mn}x \subseteq Jx$, and hence  $x \in \GmJ(M)$. 
It follows that $x \in \bigcap_{J \in \WTmI} \GmJ (M)$.

Conversely, let $x \in \bigcap_{J \in \WTmI} \GmJ (M)$. 
For $J \in \WTmI$, there exists an integer $n \geq 0$ such that $\m^{n} \subseteq \Ann (x)+J$,  hence  $J \in \tilde{W} (\m, \Ann (x))$. 
Thus we have  $\WTmI \subseteq \tilde{W} (\m, \Ann (x))$. 
It then  follows from Lemma \ref{inv-lem}  that 
\[ V(\Ann (x)) = \bigcap_{J\in \tilde{W} (\m, \Ann (x))} W(\m, J) \subseteq \bigcap_{J \in \WTmI}W(\m, J)=V(I). \]
Therefore we have $I \subseteq \sqrt{\Ann(x)}$, hence $x \in \GI(M)$.
\end{proof}

\section{Vanishing and nonvanishing theorems}

In this section we argue about the vanishing and nonvanishing of local cohomology modules with respect to  $(I, J)$. 
For the remainder of this section, we adopt the convention that $\inf~ \emptyset =\infty$  for the empty subset of  $\N$, and  $\depth~ 0=\infty$, $\dim\, 0=-1$  for the trivial $R$-module.

\begin{thm}\label{KTY1-6}
For any finitely generated $R$-module $M$ we have the
equality 
$$ 
\inf \{~ i \mid \HIJ ^i (M) \neq 0 ~ \} = \inf \{~ \depth ~ M_{\p} \mid \p \in \WIJ ~ \}.
$$

\end{thm}

\begin{proof}
We set $n =  \inf \{~ \depth ~ M_{\p} \mid \p \in \WIJ ~ \}$, and let 
$E^{\bullet} (M)$  be a minimal injective resolution of  $M$.

If $\p \in \WIJ$, then $n \leq \depth ~ M_{\p} =\inf \{~ i \mid \mu _{i} (\p , M) \neq 0 ~ \} $. 
Hence we have the equality
\begin{equation}\label{kieru}
\GIJ ( E^{i} (M) ) = \bigoplus _{\p \in \WIJ} E(R/\p) ^{\mu _{i} (\p , M)} =0, 
\end{equation}
for any integer $i<n$.
(Also note that  $\GIJ (E^n(M)) \not= 0$.)   
It follows that  $\HIJ ^{i} (M)=0$ if  $i<n$.

It suffices to show that $\HIJ ^{n} (M) \neq 0$.
We see from the equality  (\ref{kieru})  that the complex 
$\GIJ (E^{\bullet} (M))$ starts from its $n$-th term. 
Thus we have a commutative diagram
\[ \minCDarrowwidth2.5pc
\begin{CD}
 0 @>>> \HIJ ^{n} (M) @>>> \GIJ (E^{n} (M)) @>>> \GIJ (E^{n+1} (M))\\
	@.				  @.						@VVV					@VVV \\
	@.	E^{n-1} (M) @>{d^{n-1}}>> E^{n} (M) @>{d^{n}}>> E^{n+1} (M)\\
\end{CD} 
\]
with exact rows. 
Since $\Ker~d^{n}=\Im~d^{n-1} \subseteq E^{n}(M)$ is an essential extension, 
it follows that $\HIJ ^{n} (M)=\GIJ (E^{n}(M)) \cap \Ker~d^{n} \neq 0$.
\end{proof}

As a special case of the theorem, if  $J = 0$ then 
we obtain the well-known equality 
\[\inf \{~i\mid \HI ^i (M) \neq 0 ~\} = \grade ~(I, M) = \inf \{~\depth~M_{\p} \mid \p \in V(I)~\}.\]
for a finitely generated $R$-module $M$.

\begin{cor}\label{KTY1-7}
Let $M$ be a finitely generated module over a local ring  $R$  with maximal ideal $\m$. 
Then the following conditions are equivalent:
\begin{enumerate}
  \item[{(\rm 1)}]\ $M$ is $(I, J)$-torsion $R$-module.
  \item[{(\rm 2)}]\ $\HIJ ^{i} (M)=0$ for all integers $i>0$.
  \end{enumerate}
\end{cor}

\begin{proof}
We have already shown the implication $(\rm 1) \Rightarrow (\rm 2)$  in Corollary \ref{KTY1-5}(1). 

\par
\noindent 
To prove  $(\rm 2) \Rightarrow (\rm 1)$, let us denote $N=M/\GIJ(M)$. 
We only have to  show that $N=0$. 
Suppose $N \neq 0$. 
From Corollary \ref{KTY1-5} (\rm 3) and (\rm 4), 
we have  $\GIJ (N) = 0$ and $\HIJ ^{i} (N) \cong \HIJ^{i} (M) =0 $ if $i>0$. 
On the other hand, since $\m \in \WIJ $, 
the inequality  $\inf \{~ \depth ~ N_{\p} \mid \p \in \WIJ ~ \} \leq \depth N_{\m} = \depth N (< \infty)$  holds. 
Thus $\HIJ ^{i} (N) \neq 0$ for an integer $i \leq \depth~ N$ by Theorem \ref{KTY1-6}. 
This is a contradiction. 
Therefore $N=0$, and the proof is completed.
\end{proof}

\begin{thm}\label{Semi-6}
Let $M$ be a finitely generated module over a local ring  $R$. 
Suppose that $J \neq R$.  
Then $\HIJ ^{i} (M)=0$ for any $i > \dim\, M/JM$.
\end{thm}

\begin{proof}
We proceed by induction on $r=\dim\, M/JM$.
If $r=-1$, then $M=0$ by Nakayama's lemma, and  hence $\HIJ ^{i} (M)=0$ for any integer $i \geq 0$.

Now assume that $r \geq 0$. 
There is a finite filtration $0 = M_0 \subsetneq M_1 \subsetneq \cdots \subsetneq M_s = M$ of $M$ such that $M_j /M_{j-1} \cong R/{\p}_j$ for $\p_j \in \Supp (M)$ and $j=1, \ldots , s$. 
Then there are short exact sequences $0 \to M_{j-1} \to M_j \to R/{\p}_j \to 0$ for $j=1, \ldots , s$, 
and hence we have exact sequences
\[ \HIJ ^i (M_{j-1}) \to \HIJ ^i (M_j) \to \HIJ ^i (R/{\p}_j) \]
for all integers $i$ and $j$ with $i\geq 0$ and $1 \leq j \leq s$. 
Note that 
\[ \dim\, R/({\p}_j +J) \leq \dim\, R/(\Ann (M)+J) = \dim\, M/JM = r. \] 
Thus we may assume that $M=R/\P$ with $\P \in \Spec (R)$.

Since we show in  Theorem \ref{IDT} that 
$\HIJ ^i (R/\P) \cong H_{I(R/\P), J(R/\P)}^i(R/\P)$, 
replacing  $R$  by  $R/\P$, we may assume that  $R$  is an integral domain and $M =R$. 

Suppose that $\HIJ ^{\ell} (R) \neq 0$ for some integer $\ell >r$. 
We would like to derive contradiction. 
Note in this case that we have $\Ass _R (\HIJ ^{\ell} (R)) \not= \emptyset$. 

First, let us assume that  $\Ass _R (\HIJ ^{\ell} (R))$  contains a nonzero prime ideal  $\Q$. 
Then take a nonzero element $x \in \Q$. 
From the obvious short exact sequence $0 \to R \overset{x}{\to} R \to R/(x) \to 0$, one gets an exact sequence 
\[ \HIJ ^{\ell-1} (R/(x)) \to \HIJ ^{\ell} (R) \overset{x}{\to} \HIJ ^{\ell} (R). \]
Note that $\dim\, R/(J+(x))=r-1< \ell-1$, hence the induction hypothesis implies    $\HIJ ^{\ell-1} (R/(x))=0$. 
This shows that the element $x$ is $\HIJ ^{\ell} (R)$-regular. 
However, the element $x$ is in the associated prime $\Q$ of $\HIJ ^{\ell} (R)$, hence is a zero-divisor on  $\HIJ ^{\ell}(R)$. 

This contradiction forces  $\Ass _R (\HIJ ^{\ell} (R))= \{ (0) \}$. 
Note from Proposition \ref{KTY1-1} and \ref{KTY1-5} $(\rm 5)$ that 
$\Ass _R (\HIJ ^{\ell} (R)) \subseteq \WIJ$. 
Hence we have   $(0) \in \WIJ$. 
Since the set  $\WIJ$  is closed under specialization, 
 one has $\WIJ =\Spec (R)$. 
In this case one  easily sees that $\HIJ ^{\ell} (R)=0$ for any $\ell > 0$, 
which is again a contradiction.
\end{proof}

\begin{cor}\label{van for dim R/J}
Let  $R$  be a local ring and let $M$  be an $R$-module that is not necessarily finitely generated. 
Then $\HIJ ^i (M) = 0$ for any  $i > \dim\, R/J$. 
\end{cor}

\begin{proof}
Since every $R$-module is a direct limit of finitely generated submodules, 
we may write 
$M = \varinjlim _{\lambda} M _{\lambda}$  where each  $M_{\lambda}$ is a finitely generated $R$-module.
Note that if  $i > \dim\, R/J$, then  $i > \dim\, M_{\lambda}/JM_{\lambda}$. 
Therefore, by Proposition \ref{comm injlim}, 
we have 
$\HIJ ^i (M) = \varinjlim _{\lambda} \HIJ ^i (M_{\lambda}) = 0$. 
\end{proof}

Grothendieck's nonvanishing theorem says that 
the ordinary local cohomology module  $\Hm ^{r}(M)$ does not vanish whenever 
$R$ is a local ring with maximal ideal $\m$ and  $M$  is a finitely generated  $R$-module of dimension  $r$. 
The following theorem can be thought of as a generalization of this result.

\begin{thm}\label{KTY1-9}
Let   $M$  be a finitely generated module over a local ring $R$  with maximal ideal $\m $. 
Suppose that $I+J$ is an $\m $-primary ideal. 
Then we have the equality 
\[ \sup \{~ i \mid \HIJ ^{i} (M) \neq 0 ~\} = \dim\, M/JM. \]
\end{thm}

\begin{proof}
In virtue of Theorem \ref{Semi-6}, we only have to prove that $\HIJ ^{r} (M) \neq 0$ for $r=\dim\, M/JM$. 
Since $I+J$ is an $\m $-primary ideal, we have $\HIJ ^{i} (M)=\HmJ ^{i} (M)$ for any integer $i$. 
Thus we may assume that $I=\m$.
The exact sequence $ 0 \to JM \to M \to M/JM \to 0$ induces an exact sequence 
\[ \HmJ ^{r} (M) \to \HmJ^{r} (M/JM) \to \HmJ ^{r+1} (JM) .\]
We see from Theorem \ref{Semi-6} that $\HmJ ^{r+1} (JM)=0$ because $\dim\, JM/J^{2}M \leq \dim\, M/J^{2}M =\dim\, M/JM=r$. 
Furthermore, it follows from Corollary \ref{J-tor} and Grothendieck's non-vanishing theorem that  
\[ \HmJ ^{r} (M/JM)=\Hm^{r}(M/JM) \neq 0. \] 
Consequently, the exact sequence implies $\HmJ ^{r} (M)\neq 0$. 
\end{proof}

\begin{rem}
$(\rm 1)$\ 
If $J=R$, then the assertion of Theorem \ref{Semi-6} does not necessarily hold, for  $\dim\, M/JM=-1<0$ and  $\HIJ ^0 (M) \cong \GIJ (M)=M$.

\par\noindent
$(\rm 2)$\ 
If $R$ is a non-local ring, then the assertion of Theorem \ref{Semi-6} does not necessarily hold.

For example, let $R=k[x]$ be a polynomial ring over a field $k$, and set $I =(x-1)$, $J = I \cap (x) = (x^2 -x)$, and $M=R$. 
Then one has $\dim\, M/JM=0<1$ but $\HIJ ^1 (M) \neq 0$. 
\end{rem}

Even in the non-local case, one has the following result on the vanishing of local cohomology modules with respect to $(I, J)$.

\begin{thm}\label{non-local van}
Let $M$ be a finitely generated $R$-module.
Then 
 \begin{enumerate}
 \item[{(\rm 1)}]\ $\HIJ ^i (M)=0$ for all integers $i>\dim\, M$.
 \item[{(\rm 2)}]\ $\HIJ ^i (M)=0$ for all integers $i>\dim\, M/JM+1$.
 \end{enumerate}
\end{thm}

\begin{proof}
$(\rm 1)$ 
This easily follows from Theorem \ref{dir-lim} and Grothendieck's vanishing theorem. 

\par\noindent
$(\rm 2)$ 
We prove this by induction on $r=\dim\, M/JM$. 
When $r=-1$, Nakayama's Lemma says that $(1+a)M=0$ for some $a\in J$. 
Hence we have $Jx=Rx$ for any $x\in M$, which implies that the $R$-module $M$ is $(I, J)$-torsion. 
Proposition \ref{KTY1-5} $(\rm 1)$ shows that $\HIJ ^i (M)=0$ for every $i>0=r+1$, as desired. 
When $r\ge 0$, we can prove the assertion along the lines as in the proof of Theorem \ref{Semi-6}. 
\end{proof}

As one of the main theorems of this section, we shall prove a generalization of Lichtenbaum-Hartshorne theorem in Theorem \ref{GLHVT}. 
For this we begin with the following lemma.

\begin{lem}\label{lem-LH}
Let  $n$ be a nonnegative integer.
Suppose that $\HIJ^{i}(R)=0$ for all $i>n$. 
Then the following hold for any $R$-module $M$ which is not necessarily finitely generated. 
\begin{itemize}
\item[(1)]
$\HIJ^{i}(M)=0$ for all $i>n$. \vspace{4pt} 
\item[(2)]
$\HIJ ^n (M) \cong \HIJ ^n (R) \otimes _R M$
\end{itemize}
\end{lem}

\begin{proof}
First we should note that, by virtue of Proposition \ref{comm injlim}, we only have to prove the lemma for a finitely generated $R$-module $M$. 

$(1)$  
We have shown in the previous theorem that $\HIJ^{i}(M)=0$ if $i> \dim\, M$. 
We prove the assertion by descending induction on $i$. 
There exists a short exact sequence 
\[ 0 \to N \to R^{m} \to M \to 0\]
where $m$ is an integer and $N$ is a finitely generated $R$-module. 
This sequence induces an exact sequence 
\[ \HIJ^{i}(R^{m})\to \HIJ^{i}(M) \to \HIJ^{i+1}(N). \]
By the induction hypothesis, the equality $\HIJ^{i+1}(N)=0$  holds. 
Thus we see that $\HIJ^{i}(M)=0$.

\par \noindent
$(2)$ 
By claim $(1)$, the functor  $\HIJ ^n$  is a right exact functor on the category of $R$-modules, hence it is represented as a tensor functor. 
\end{proof}

For an $R$-module $M$, we set 
\[ \Assh _{R} (M)=\{~\p \in \Ass _R (M)\mid \dim\, R/\p =\dim _R\, M~\}.\]

We are now ready to prove the generalized version of Lichtenbaum-Hartshorne theorem.

\begin{thm}\label{GLHVT}
Let  $(R, \m)$ be a local ring of dimension $d$, and let $I$ and $J$ be proper ideals of $R$. 
Then the following conditions are equivalent. 
\begin{enumerate}
	\item[{(\rm 1)}]\ $\HIJ ^{d} (R)=0.$
	\item[{(\rm 2)}]\ For each prime ideal $\p \in \Assh (\hat{R})$ with $J \hat{R} \subseteq \p$, 
	we have $\dim\, \hat{R}/(I \hat{R}+\p)>0$.
\end{enumerate}
\end{thm}

\begin{proof}
$(\rm 1) \Rightarrow (\rm 2)$ 
Suppose that $\HIJ^{d} (R)=0$, 
and that there exists $\p \in \Assh (\hat{R})$  satisfying   
$J \hat{R} \subseteq \p$ and $\dim\, \hat{R}/(I\hat{R} +\p)=0 $. 
We would like to derive a contradiction. 

By Lemmma \ref{lem-LH} we have 
 $\HIJ ^{d} (\hat {R}/\p) =  0$.
On the other hand, since  $J \subseteq \p$, $\hat{R}/\p$  is a $J$-torsion module over  $R$. 
Hence Corollary \ref{J-tor} implies that  $\HIJ ^{d} (\hat {R}/\p) \cong  H _I ^d (\hat{R}/\p)$, which is isomorphic to  $H_{I(\hat{R}/\p)} ^d(\hat{R}/\p)$.
Note here that $(\hat{R}/\p, \m\hat{R}/\p)$ is a $d$-dimensional complete local ring and $(I\hat{R}+\p)/\p$ is $\m \hat{R}/\p$-primary ideal.
Thus we have
$\HIJ ^{d} (\hat{R}/\p) \cong H_{\m (\hat{R}/\p)}^d (\hat{R}/\p)$, which is nonzero by Grothendieck's nonvanishing theorem. 
This is a contradiction.

$(\rm 2) \Rightarrow (\rm 1)$ 
Suppose that $\HIJ^{d} (R) \neq 0$, we shall show a contradiction under the condition $(2)$.
Since $\hat{R}$ is faithfully flat, 
it holds by Lemma \ref{lem-LH} that
$$
\HIJ ^{d}(\hat {R})=\HIJ^{d}(R) \otimes _{R} \hat{R}\neq 0. 
$$
Considering a filtration of ideals of  $\hat{R}$;   
\[ 0=K_{0}\subsetneq  K_{1}\subsetneq \cdots \subsetneq K_{s-1}\subsetneq K_{s}=\hat{R}, \]
with  $K_{j}/K_{j-1} \cong \hat{R}/\p_{j}$ for prime ideals  $\p _{j}$ of $\hat{R}$  for $1\leq j\leq s$, 
we see that there is at least one prime ideal $\p$  of  $\hat{R}$ such that  $\HIJ ^{d} (\hat{R}/\p)\neq 0$. 

First consider the case that  $J \subseteq \p$. 
Then, since  $\hat{R}/\p$  is a $J$-torsion $R$-module, 
 it follows from Corollary \ref{J-tor} that  $\HIJ ^{d} (\hat{R}/\p) = H_I^d (\hat{R}/\p) = H_{I(\hat{R}/\p)} ^d (\hat{R}/\p)$.  
If  $\dim\, \hat{R}/\p <d$, then $ H_{I(\hat{R}/\p)} ^d (\hat{R}/\p) = 0$ by Grothendieck's vanishing theorem, and this is a contradiction. 
If  $\dim\, \hat{R}/\p =d$, then $\p \in \Assh (\hat{R})$,  hence  $\dim\, (\hat{R}/I\hat{R} + \p) >0$ by the assumption (2). 
Thus we have  $H_{I(\hat{R}/\p)} ^d (\hat{R}/\p) = 0$  by the Lichtenbaum-Hartshorne theorem. 
This is again a contradiction. 

Next consider the case  $J \not\subseteq \p$. 
Denote  $\bar{R} = R/\p \cap R$. 
Applying Theorem \ref{IDT} to the natural projection  $R \to \bar{R}$, we have  $\HIJ ^{d} (\hat{R}/\p)  =  H_{I\bar{R}, J\bar{R}} ^{d} (\hat{R}/\p)$. 
Since  $\dim\, \bar{R}/J\bar{R} < \dim\, \bar{R} \leq d$, 
it follows from Corollary \ref{van for dim R/J} that 
$H_{I\bar{R}, J\bar{R}} ^{d} (\hat{R}/\p) = 0$, which is  a  contradiction as well.  
\end{proof} 

\begin{rem}
In \cite[Theorem 1.1]{DNT} it is proved that the first condition in Theorem \ref{GLHVT} is equivalent to the condition that for each $\p\in\Assh(\hat{R})$ there exists $\q\in\WIJ$ with $\dim\,\hat{R}/(\q\hat{R}+\p) >0$.
We see that this condition implies the second condition in Theorem \ref{GLHVT}, but the opposite implication seems not obvious.
(The authors do not know how to prove the opposite implication directly.)
The point is that the second condition in Theorem \ref{GLHVT} is concerning the ideals $I$ and $J$, but not concerning the set $\WIJ$.
\end{rem}

Recall that the arithmetic rank of an ideal  $I$, denoted by $\ara(I)$, is defined to be the least number of elements of $R$ required to generate an ideal which has the same radical as $I$.

\begin{prop}
Let $M$ be an $R$-module. 
Then $\HIJ ^{i} (M)=0$ for any integer $i>\ara (I\bar {R})$,  where $\bar {R} =R/\sqrt{J+\Ann (M)}$.
\end{prop}

\begin{proof}
Denote  $R^{\prime} =R/\Ann _{R} (M)$. 
Then  $\bar {R} = R^{\prime}/\sqrt{JR^{\prime}}$ and $\Ann _{R^{\prime}} (M)=0$.  Since we have an isomorphism $\HIJ ^{i}(M)\cong H^{i}_{IR^{\prime}, JR^{\prime}} (M)$  by Theorem \ref{IDT},  we may assume that $\Ann _{R} (M)=0$.

Let us denote $s =\ara (I\bar{R})$. 
Then we find  a sequence   $\aa = a_1, a_2, \ldots , a_s$ of $s$ elements  in  $R$ such that  $\sqrt{I\bar{R}}=\sqrt{\aa \bar{R}}$. 
Then it is easy to see from Proposition \ref{replace IJ} that the equality 
$$
\HIJ ^{i} (M) = H^{i}_{\aa R, J} (M) = 
H^{i} ( C^{\bullet }_{\aa, J} \otimes M) 
$$
holds for any $i$. 
Since the complex $C^{\bullet }_{\aa, J}$ is of length $s$, 
we see that $H^{i}(C^{\bullet }_{\aa, J} \otimes M)=0$  for all integers $i>s=\ara (I\bar{R})$.
\end{proof}

\section{The Local duality theorem and other functorial isomorphisms}

For a local ring $R$ with maximal ideal $\m$, 
we denote the functor $\Hom_{R}(-, E_{R}(R/\m))$ by $(-)^ {\vee} $. 
Let  $(R, \m)$ be a $d$-dimensional Cohen-Macaulay complete local ring.
Then it is well known that it holds the local duality theorem, which states 
 the existence of functorial isomorphisms 
$$
\Hm^{d-i}(M)^{\vee} \cong \Ext^{i}_{R}(M, K_{R}), 
$$
for a finitely generated $R$-module $M$  and any integer $i\geq 0$. 
Note that  $K_{R}$ is the canonical module of $R$ given as 
 $K_{R} = \Hm^{d}(R)^{\vee}$. 
The following theorem is thought of as a generalization of the local duality theorem.

\begin{thm}\label{LD}
Let $(R, \m)$ be a Cohen-Macaulay complete local ring of dimension $d$, 
and let  $J$  be a perfect ideal of  $R$  of  grade  $t$, i.e. $\pd _R R/J = \grade (J, R) =t$. 
Then, for a finitely generated $R$-module $M$, there is a functorial isomorphism \[ \HmJ ^{d-i} (M)^{\vee} \cong \Ext_{R}^{i-t} (M, K) \]
for all integer $i$, where $K=\HmJ ^{d-t} (R) ^{\vee}$.
\end{thm}

To prove the theorem we need the following lemma.

\begin{lem}\label{new intersection}
Let  $R$  be a Cohen-Macaulay local ring of dimension $d$ and let $J$  be a perfect ideal of  $R$  of grade  $t$. 
Then  the inequality  $\height~\p \geq d-t$ holds for any $\p \in W(\m, J)$.     \end{lem}

\begin{proof}
If  $\p + J$  is an $\m$-primary ideal, then $R/\p \otimes _R R/J$ is of finite length, hence the new intersection theorem \cite{Hochster, Pskine-Szpiro, Roberts} implies that  $\dim\, R/\p \leq \pd _R R/J = t$ therefore  $\height~\p \geq d - t$.
\end{proof}

Now we proceed to the proof of Theorem \ref{LD}.

\begin{proof}
Let us denote  $T^{i}(-)=\HmJ ^{d-t-i}(-)^{\vee}$, and we shall show the isomorphism of functors  $T^{i}(-) \cong \Ext_{R}^{i}(-, K)$.

Note that  $R/J$  is a Cohen-Macaulay ring of dimension $d-t$. 
Hence we see from Corollary \ref{van for dim R/J} that $\HmJ^{d-t}(-)$ is a right exact functor on the category of all $R$-modules. 
Note from Lemma \ref{lem-LH} that there is a natural isomorphism 
 $M\otimes_{R} \HmJ^{d-t}(R) \cong \HmJ^{d-t}(M)$ for any $R$-module $M$. 
Thus we have 
\[ T^{0}(M) \cong (M\otimes \HmJ^{d-t}(R))^{\vee} \cong \Hom (M, K).\]

Let $0\to L\to M\to N \to 0$ be an exact sequence of $R$-modules. 
Then we have a long exact sequence 
\[ \cdots \to \HmJ^{d-t-1}(N) \to \HmJ^{d-t}(L) \to \HmJ^{d-t}(M) \to \HmJ^{d-t}(N) \to 0, \]
which induces a long exact sequence
\[ 0\to T^{0}(N) \to T^{0}(M) \to T^{0}(L) \to T^{1}(N) \to \cdots .\]

Therefore the proof will be completed if we show that $T^{i} (F)=0$ 
for any integer $i>0$ and any free $R$-module $F$.
It is enough to show that $\HmJ ^{d-t-i}(R)=0$ for $i>0$.
If $\p \in \WmJ$, then we have $\depth R_{\p} = \height~ \p \geq d-t$ by Lemma \ref{new intersection}.  
Thus we see from Theorem \ref{KTY1-6} that $\HmJ^{j} (R)=0$ for all integer $j<d-t$. 
\end{proof}

\begin{rem}
We should note that $K = H_{\m, J}^{d-t} (R)^{\vee}$ in the theorem is not necessarily  a finite $R$-module, even if  $R$  is a Gorenstein ring. 

In fact, when $R$  is Gorenstein, we shall show below in Proposition \ref{intersection}  the following equality 
$$
\Ass (\HmJ ^{d-t}(R))=\{~\p \in \WmJ \mid \height~ \p =d-t ~\}.  
$$ 
This set is not equal to  $\{ \m \}$  if  $t$  is positive. 
In this case,  $\HmJ ^{d-t} (R)$  is not an artinian $R$-module, hence 
$K$  is not a noetherian $R$-module. 
\end{rem}

Let  $R$  be a local ring with maximal ideal $\m$. 
Then we shall see in this section that there often exist dualities between 
local cohomology with respect to $(\m, J)$  and ordinary local cohomology with support in $J$.

For an $R$-module $M$ and an ideal $J$ of $R$, 
we denote by $M^{\widehat{\hphantom{J}}}_J$  the $J$-adic completion of $M$, which is  defined to be the projective limit $\varprojlim _n M/J^nM$.  
The following theorem should be compared with the result of Greenlees-May \cite{Greenlees-May}.

\begin{thm}
Let  $R$  be a Cohen-Macaulay local ring of dimension $d$ with canonical module  $K_R$. 
And let   $J$  be an ideal of  $R$  with  $\dim\, R/J = d-r$. 
Then there is a natural isomorphism 
\[ \HmJ^{d-r}(R)^{\widehat{\hphantom{J}}} _J \cong H_{J}^{r}(K_R)^{\vee}. \]
\end{thm}

\begin{proof}
We have the following isomorphisms  
\begin{align*}
 \HmJ^{d-r}(R)/J^{n}\HmJ^{d-r}(R)  
	&\cong \HmJ^{d-r}(R) \otimes _R R/J^{n} & \\
	&\cong \HmJ^{d-r}(R/J^{n}) &\text{(by Lemma \ref{lem-LH})\hphantom{ry}} \\
	&\cong \Hm ^{d-r}(R/J^{n}) &\text{(by Corollary \ref{J-tor})} \\
	&\cong \Ext_{R}^{r} (R/J^{n},K_R)^{\vee},   &
\end{align*}
where the last isomorphism follows from the local duality theorem applied to the $R$-module  $R/J^{n}$. 
Since these isomorphisms are functorial, taking project limits we have the isomorphism 
$$\HmJ^{d-r}(R)^{\widehat{\hphantom{J}}}_J 
\cong \varprojlim_{n \in \N} (\Ext_{R}^{r} (R/J^{n}, K_R)^{\vee} ).
$$
On the other hand, it follows from the definition of ordinary local cohomology that 
$$
H_{J}^{r}(K_R)^{\vee}
\cong (\varinjlim_{n \in \N} (\Ext_{R}^{r} (R/J^{n}, K_R))^{\vee} 
\cong \varprojlim_{n \in \N} (\Ext_{R}^{r} (R/J^{n}, K_R)^{\vee} ). 
$$
Combining these isomorphisms we finish the proof of the theorem. 
\end{proof}

\begin{rem}
It is natural to ask whether there is a functorial  isomorphism 
$$
H_{\m, J} ^{d-i}(R)^{\widehat{\hphantom{J}}}_J \cong H_J^{i}(K_R)^{\vee}.
$$ 
for any integer $i$. 

This is however not true.
For example, let $R=k[[X, Y, Z,W]]$, and $J=(X, Y) \cap (Z, W)$. 
Then it is easy to see that $H_J^{3}(R) = H_{\m}^4(R) = E_{R}(R/\m)$, but $H_{\m, J}^{1}(R)=0$. Thus $H_{\m, J}^{1}(R)^{\widehat{\hphantom{J}}}_J \not \cong H_J^{3}(R)^{\vee}$.
\end{rem}

\begin{prop}\label{intersection}
Let  $R$  be a Cohen-Macaulay local ring of dimension $d$ with canonical module  $K_R$. 
Assume that  $J$  is a perfect ideal of grade  $t$. 
Then the following equality holds. 
$$
\Ass (\HmJ ^{d-t}(K_R))=\{~\p \in \WmJ \mid \height~ \p =d-t ~\} 
$$ 
\end{prop}

\begin{proof}
Let  $E^{\bullet}$  be a minimal injective resolution of the $R$-module  $K_R$.  Then it is known that  $E^i = \bigoplus _{\substack{\height \p = i\\ \p \in \Spec R}}E(R/\p)$, hence  $\Gamma _{\m, J} (E^i) = \bigoplus _{\substack{\height \p = i\\ \p \in \WmJ}}E(R/\p)$. 
Therefore by Lemma \ref{new intersection}, there is a short exact sequence 
$$
0 \to \HmJ^{d-t}(K_R) \to \bigoplus _{\substack{\height \p =d-t\\ \p \in \WmJ}}E(R/\p)  \to \bigoplus _{\substack{\height \p =d-t+1 \\ \p \in \WmJ}} E(R/\p).
$$
This implies that  $\Ass  (\HmJ ^{d-t} (K_R)) \subseteq \{ \p \in \WmJ \ | \ \height ~\p = d-t \}$. 
Conversely, let  $\p \in \WmJ$ be a prime with $\height~ \p =d-t$. 
Then by the above exact sequence, we see 
$$
(\HmJ ^{d-t}(K_R))_{\p}= E_{R_{\p}}(\kappa (\p)) \supseteq \kappa (\p).
$$
Therefore $\p \in \Ass (\HmJ ^{d-t} (K_R))$. 
\end{proof}

We recall that  the generalized local cohomology in the sense \cite{He} is defined as
$$
H _J ^i(M, N) = \varinjlim_{n} \Ext_R^i (M/J^nM, N), 
$$  
for $R$-modules $M$ and $N$, and for $i \geq 0$.

\begin{thm}\label{t2}
Let $(R, \m)$ be a Gorenstein local ring of dimension $d$, 
which is $J$-adically complete. 
Then there is an isomorphism 
\[ \GmJ (M) \cong H_J ^d (M, R)^{\vee},  \]
for any finitely generated $R$-module $M$. 
\end{thm}

\begin{proof}
From the definition and the local duality theorem we have the following isomorphisms and inclusion. 
\[
\begin{array}{rcl}
H_J ^d (M, R)^{\vee} & = & (\varinjlim _n \Ext _R ^d (M/J^n M, R))^{\vee} \vspace{4pt}\\
& \cong & \varprojlim \varGamma _{\m} (M/J^n M) \vspace{4pt}\\
& \hookrightarrow & \varprojlim M/J^n M \vspace{4pt}\\
& \cong & M
\end{array}
\]
We would like to show that the image of the composite map 
$f:H_J ^d (M, R)^{\vee} \hookrightarrow M$ above is equal to $\GmJ (M)$. 

Let $y \in \Im f$.
Applying the Artin-Rees lemma, we see that $J^m M \cap Ry \subseteq Jy$ for some integer $m>0$.
On the other hand, it follows from the choice of $y$ that the image of $y$ in $M/J^n M$ belongs to $\varGamma _{\m} (M/J^n M)$ for each $n>0$. 
Hence we have $\m ^\ell y\subseteq J^m M$ for some $\ell>0$. 
Thus we get $\m ^\ell y\subseteq J^m M\cap Ry\subseteq Jy$, that is, $y\in\GmJ (M)$. 

Conversely, let $y \in \GmJ (M)$.
Then $\m ^m y\subseteq Jy$ for an integer $m>0$. 
Hence we have $\m ^{mn} y\subseteq J^n y\subseteq J^n M$ for any $n>0$. 
Therefore for each $n>0$ the image of $y$ in $M/J^n M$ belongs to 
$\varGamma _{\m} (M/J^n M)$, which says that $y \in \Im f$.
\end{proof}


Before proving further results, we make a number of
preparatory remarks about the local cohomologies of the canonical dual of a module. 
  
Suppose that $R$ admits the dualizing complex $D_R$. 
We denote by $K_M $ the canonical module of an $R$-module  $M$, which is defined to be 
$$
K_M =H^{d-r} (\RHom _R (M, D_R )),
$$
where $d=\dim\, R$ and $r=\dim\, M$. 
Note that in case $R$ is a Gorenstein ring we have $K_M = \Ext _R^{d-r}(M, R)$. 
Therefore if  $R$  is Gorenstein and if  $r=d$, then  $K_M$  equals the ordinary dual  $M^* = \Hom _R(M, R)$.

Remember that for an integer $n\geq 0$, we say that $M$ satisfies the condition $(S_n)$ provided 
$$
\depth _{R_{\p}} M_{\p} \geq\inf\{ n, \dim _{R_{\p}}\, M_{\p} \}.
$$ 
for all $\p \in \Spec(R)$

\begin{lem}\label{t1}
Let $R$ be a Gorenstein local ring of dimension $d$, 
and $M$ a finitely generated $R$-module of dimension $r$. 
Suppose that $\Ass _{R} M=\Assh _{R} M$ and that $M$ satisfies $(S_{n+1})$ for some $n\geq 0$. 
Then there is an isomorphism 
\[ H_J ^{r-i} (K_M)\cong H_J ^{d-i} (M, R) \]
for all $0\leq i\leq n$. 
\end{lem}

\begin{proof}
Since a module satisfying $(S_i)$ also satisfies $(S_{i-1})$, it is enough to show that $H_J ^{r-n} (K_M) \cong H_J ^{d-n} (M, R)$.
Note that $\grade _R M=d-r$. 
Take a maximal $R$-sequence $\yy =y_1 , y_2 , \ldots , y_{d-r}$ in $\Ann _{R} M$. 
Replacing $R$ by $R/\yy R$, we may assume that $r=d$. 

Let $S=R \ltimes M$ be the trivial extension of $R$ by $M$. 
Since $S$ is isomorphic to $R \oplus M$ as an $R$-module, $K_S$ is isomorphic to $K_R \oplus K_M$ as an $R$-module. 
This induces natural isomorphisms
$$
\begin{cases}
H_J ^{d-n} (K_S)\cong H_J ^{d-n} (R)\oplus H_J ^{d-n} (K_M),\\ 
H_J ^{d-n} (S, R)\cong H_J ^{d-n} (R)\oplus H_J ^{d-n} (M, R).
\end{cases}
$$
Thus we have only to show that $H_J ^{d-n} (K_S)\cong H_J ^{d-n} (S, R)$. 

There are isomorphisms
\begin{align*}
H_J ^{d-n} (S, R) & =\varinjlim _k \Ext _R ^{d-n} (S/J^k S, R)\\
& \cong \varinjlim _k \Ext _S ^{d-n} (S/J^k S, \RHom _R (S, R))\\
& \cong \varinjlim _k \Ext _S ^{d-n} (S/J^k S, D_S )\\
& = H_J ^{d-n} (D_S ).
\end{align*}
There is a chain map $H_J ^{d-n} (K_S )\to H_J ^{d-n} (D_S )$ 
induced by the augmentation $K_S =H^0 (D_S)\to D_S $. 
We have to show that this map is an isomorphism. 
Since we have a spectral sequence
\[ E_2 ^{pq} =H_J ^p (H^q (D_S)) \Rightarrow H_J ^{p+q} (D_S ), \]
it suffices to show that $\dim_R\,\Ext _R ^q (S, R)<d-n-q$ for any $q>0$. 

Let us show that the $R$-module $S$ satisfies $(S_{n+1})$. 
Take $\p \in \Supp _R\,S$. 
We want to prove $\depth _{R_{\p }} S_{\p } \geq \inf \{ n+1, \dim _{R_{\p }}\, S_{\p }\} $.
Because $S_{\p } \cong R_{\p } \oplus M_{\p }$ as an $R_{\p }$-module, 
we have $\depth _{R_{\p}} S_{\p} =\inf \{\depth\,R_{\p} , \depth _{R_{\p}} M_{\p}\} 
=\depth _{R_{\p}} M_{\p} \geq \inf \{ n+1, \dim _{R_{\p}}\, M_{\p} \}$. 
It is easy to see that $\dim _{R_{\p}}\, S_{\p} =\dim\, R_{\p} =\dim _{R_{\p}}\, M_{\p}$ since $\Ass_R\,M=\Assh_R\,M$ and $\dim_R\,M=r=d$. 
Thus $S$ satisfies $(S_{n+1})$. 

Suppose that $\dim _R\, \Ext _R ^q (S, R)\geq d-n-q$ for some $q>0$.
Then there exists $\p \in \Supp _R \Ext _R ^q (S, R)$ such that $\dim\, R/\p \geq d-n-q$.
Hence we have $\Ext _{R_{\p}} ^q (S_{\p}, R_{\p}) \neq 0$ and $\height~ \p \leq n+q$.
The local duality theorem yields an isomorphism 
$H_{\p R_{\p}} ^{\height \p -q}(S_{\p}) \cong \Ext _{R_{\p}} ^q (S_{\p}, R_{\p})^{\vee} \neq 0$,
and so $\depth _{R_{\p}} S_{\p} \leq \height \p -q \leq n$.
Since $S$ satisfies $(S_{n+1})$, 
we have $\depth _{R_{\p}} S_{\p}=\dim _{R_{\p}}\, S_{\p} =\dim\, R_{\p} =\height \p$.
Therefore we must have $q\leq 0$, a contradiction.
This contradiction completes the proof of the lemma.
\end{proof}

Let $R$ be a Gorenstein local ring of dimension $d$, $J$ an ideal of $R$, 
and $M$ a finitely generated $R$-module of dimension $r$. 
Then we have  $K_M = \Ext _R^{d-r}(M, R)$. 
Thus it is easy to see that $\dim\, K_M =\dim\, M=r$ and $\Ass~ K_M =\Assh~ K_M$. 
Moreover,  $K_M$ satisfies $(S_2)$. 
Hence by Lemma \ref{t1}, we obtain
\[ H_J ^{r-i} (K_{K_M}) \cong H_J ^{d-i} (K_M, R) \]
for $i=0, 1$. 
On the other hand, the following lemma holds.

\begin{lem}\label{p1}
Let $R$ be a local ring having the dualizing complex $D_R $, and let 
$M$ be a finitely generated $R$-module of dimension $r$. 
Then
\[ H_J ^r (K_{K_M})\cong H_J ^r (M). \]
\end{lem}

\begin{proof}
In virtue of \cite[Theorem\,1.2]{Kawasaki} we can take a Gorenstein ring $A$ of dimension $r$ with a surjective ring homomorphism $\phi :A \to R/\Ann\, M$. 
Replacing $R$ (resp. $J$) with $A$ (resp. $\phi ^{-1} (J(R/\Ann\,M))$), we may assume that $R$ is an $r$-dimensional Gorenstein local ring. 
Then we have $K_M\cong M^{\ast}$ and $K_{K_M}\cong M^{\ast \ast}$, where $(-)^{\ast}=\Hom _R (-, R)$. 
Let $f:M\to M^{\ast\ast}$ be the natural homomorphism.
It follows from \cite[Proposition\,2.6]{AB} that $\Ker\, f \cong \Ext _R ^1 (\tr M, R)$ and 
$\Coker\,f\cong \Ext _R ^2 (\tr M, R)$, where $\tr M$ denotes the Auslander transpose of $M$. 
It is easily seen that $\dim _R\, \Ext _R ^i (X, R) \leq r-i$ for any finitely generated $R$-module $X$ and $i\geq 0$. 
Hence $\dim\,(\Ker\,f) \leq r-1$ and $\dim\,(\Coker\,f) \leq r-2$. 
From this one sees that the induced homomorphism $H_J^r(f):H_J ^r (M)\to H_J ^r (M^{\ast\ast})$ is an isomorphism.
\end{proof}

Combining the isomorphisms given in Lemmas \ref{t1} and \ref{p1}, 
we conclude that the following corollary holds.

\begin{cor}\label{c1}
Let $R$ be a Gorenstein local ring of dimension $d$, 
and let  $M$ be a finitely generated $R$-module of dimension $r$. 
Then there is an isomorphism 
\[ H_J ^r (M)\cong H_J ^d (K_M, R).  \]
\end{cor}


Now we shall show  the following theorem, which is essentially shown in \cite{Schenzel}.
We should note that it holds without assuming that the local ring $R$ is Gorenstein.

\begin{thm}\label{c2}
Let $(R, \m)$ be a complete local ring and let $M$ be a finitely generated $R$-module of dimension $r$.
Then we have an isomorphism 
\[ H_J ^r (M)^{\vee} \cong \GmJ (K_M ). \]
\end{thm}

\begin{proof}
Since $R$ is a complete local ring, there exists a Gorenstein complete local ring  $S$  of $\dim\, S = \dim\, R =d$  with a surjective ring homomorphism  $\phi :S\to R$. 
Set $\a =\phi ^{-1} (J)$. 
Let us denote the maximal ideal of $S$  by $\n$. 
Note that  $S$  is $\a$-adically complete as well. 
Thus we can apply Corollary \ref{c1}, Theorem \ref{t2} and Theorem \ref{IDT}, and we obtain the following isomorphisms. 
\[
\begin{array}{rcll}
H_J ^r (M)^{\vee}
 & = & \Hom _R (H_{\a} ^r (M), E_R (R/\m )) & \vspace{4pt}\\
 & \cong & \Hom _S (H_{\a} ^r (M), E_S (S/\n )) & \vspace{4pt}\\
 & \cong & \Hom _S (H_{\a} ^d (K_M, S), E_S (S/\n )) \hphantom{KK} &\text{(by Corollary \ref{c1})} \vspace{4pt}\\
 & \cong & \varGamma _{\n, \a} (K_M) &\text{(by Theorem \ref{t2})} \vspace{4pt}\\
 & \cong & \GmJ (K_M)  &\text{(by Theorem  \ref{IDT}).} 
\end{array}
\]
\end{proof}

\begin{cor}\label{c2cor}
As in the previous theorem, let $(R, \m)$ be a complete local ring and let $M$ be a finitely generated $R$-module of dimension $r$.
Then we have an isomorphism 
\[ \HmJ^{r}(M)^{\vee}=\G_{J}(K_{M}). \]
\end{cor}

\begin{proof}
We know from Proposition \ref{inv-lim} that the equality 
$$\G_{J}(K_{M}) = \varprojlim_{I\in \WTmJ}\G_{\m, I}(K_{M})$$ holds.
Therefore it follows from the previous theorem that 
$$
\G_{J}(K_{M})
=\varprojlim_{I \in \WTmJ}\left(\HI^{r}(M)^{\vee}\right) 
=\left(\varinjlim_{I \in \WTmJ} \HI^{r}(M)\right)^{\vee}. 
$$
The last module is isomorphic to 
$\HmJ^{r}(M)^{\vee}$  by Theorem \ref{dir-lim}. 
\end{proof}

\section{Derived functors  on derived categories}

We denote by  $D^b (R)$  the derived category consisting of all bounded complexes over  $R$. 
The left exact functor $\Gamma _{I, J}$  defined on the category of  $R$-modules induces the right derived functor  $\RG_{I, J} : D^b(R) \to D^b (R)$. 
In this section we show several isomorphisms between functors involving  $\RG_{I, J}$.

\begin{lem}\label{trivial iso}
Let  $X, Y \in D^b (R)$. 
Then there are natural isomorphisms in  $D^b (R)$. 
$$
 X \otimes _R^{\L} \RG _{I, J} (Y) \cong 
 \RG _{I, J} ( X \otimes _R^{\L}  Y) \cong 
 \RG _{I, J} (X) \otimes _R^{\L} Y.  
$$
\end{lem}

\begin{proof}
Let  $\aa$  be a sequence of elements of  $R$  which generate  $I$. 
Then all these complexes are isomorphic to 
$X \otimes _R^{\L} (C _{\aa, J} \otimes _R ^{\L} Y) 
\cong   C _{\aa, J} \otimes _R ^{\L} (X  \otimes _R^{\L} Y)  
\cong   (C _{\aa, J} \otimes _R ^{\L} X)  \otimes _R^{\L} Y. 
$
\end{proof}

\begin{thm}\label{p-38}
Let $(R, \m)$ be a $d$-dimensional complete local ring admitting the dualizing complex $D_R$, 
and let  $X$ be a bounded $R$-complex with finitely generated homologies. 
Suppose that $J \subseteq \sqrt{I}$, 
then there is an isomorphism
\[ \RG_{I} (X) \cong \RG_{I}(\RG_{\m, J} (\RHom (X, D_{R}))^{\vee})[-d]. \]
\end{thm}

\begin{proof}
Since $\RHom (X,D_{R})$ is a bounded $R$-complex with finitely generated homologies 
and $X \cong \RHom(\RHom (X,D_{R}),D_{R})$, 
it is enough to show that   
\[ \RG_{I} (\RHom (X, D_{R})) \cong \RG_{I}(\RG_{\m, J} (X)^{\vee})[-d]. \]
Note from the local duality theorem that there is an isomorphism 
$
\RHom (X,D_{R})[d] \cong \RG_{\m}(X)^{\vee} 
$ 
in  $D^b(R)$.  
Therefore we only have to show that 
\[ \RG_{I} (\RG_{\m}(X)^{\vee}) \cong \RG_{I}(\RG_{\m, J} (X)^{\vee}).\]
From the definition of  $\RG_I$  we have an isomorphism 
\begin{align*}
\RG_{I} (\RG_{\m}(X)^{\vee}) 
	&\cong \varinjlim_{n} \RHom(R/I^{n}, \RG_{\m}(X)^{\vee})\\
	&\cong \varinjlim_{n} ((R/I^{n}\otimes_{R}^{\L} \RG_{\m}(X))^{\vee}), 
\end{align*}
and similarly 
\[\RG_{I}(\RG_{\m, J} (X)^{\vee}) 
	\cong \varinjlim_{n} ((R/I^{n}\otimes_{R}^{\L} \RG_{\m,J}(X))^{\vee}).\]
Thus the proof will be completed if we show that there is a natural isomorphism  $$
R/I^{n}\otimes_{R}^{\L} \RG_{\m}(X) 
\cong R/I^{n}\otimes_{R}^{\L} \RG_{\m,J}(X).
$$
In virtue of Lemma \ref{trivial iso}, this is equivalent to 
$$
\RG _{\m} (R/I^{n}) \otimes_{R}^{\L} X \cong \RG _{\m, J}( R/I^{n}) \otimes_{R}^{\L} X.
$$
Therefore it is enough to show that  
$\RG _{\m} (R/I^{n}) \cong \RG _{\m, J}( R/I^{n})$. 
But this is trivial, since   $R/I^n$  is a $J$-torsion module. 
\end{proof}

\begin{thm}\label{p-53}
Let $(R, \m)$ be a $d$-dimensional complete local ring with dualizing complex $D_R$, and let  $X$ be a bounded $R$-complex with finitely generated homologies. 
Then there is an isomorphism
\[ \RG_{\m, J} (X) \cong \RG_{\m, J}(\RG_{J} (\RHom (X, D_{R}))^{\vee})[-d]. \]
\end{thm}

\begin{proof}
Similarly as in the proof of Theorem \ref{p-38}, 
it is enough to show that 
\[ \RG_{\m, J}(\RG_{\m}(X)^{\vee})\cong \RG_{\m, J}(\RG_{J}(X)^{\vee}). \]
For each ideal $I \in \WTmJ$ and for an integer $n\geq 1$, we have the following isomorphisms hold by Lemma \ref{trivial iso}.  
\begin{align*}
R/I^{n} \otimes_{R}^{\L} \RG_{J}(X)
&\cong \RG_{J}(R/I^{n})\otimes_{R}^{\L}X\\
&\cong \RG_{I+J}(R/I^{n})\otimes_{R}^{\L}X\\
&\cong \RG_{\m}(R/I^{n})\otimes_{R}^{\L}X\\
&\cong R/I^{n}\otimes_{R}^{\L}\RG_{\m}(X)\\
\end{align*}
Hence, 
\begin{align*}
\RHom (R/I^{n},\RG_{J}(X)^{\vee}) 
	&\cong (R/I^{n}\otimes_{R}^{\L} \RG_{J}(X))^{\vee}\\
	&\cong (R/I^{n}\otimes_{R}^{\L}\RG_{\m}(X))^{\vee}\\
	&\cong \RHom (R/I^{n}, \RG_{\m}(X)^{\vee}).
\end{align*}
Applying the functor 
$\lim\limits_{\substack{\longrightarrow \\ I\in \WTmJ}} (\lim\limits_{\substack{\longrightarrow \\ n\in \N}}(-))$, 
we see from Theorem \ref{dir-lim} that 
\[ \RG_{\m, J}(\RG_{J}(X)^{\vee}) \cong \RG_{\m, J}(\RG_{\m}(X)^{\vee}). \]
\end{proof}

As a natural extension of terminology, we say that a complex  $X \in D^b(R)$ is  $(I, J)$-torsion if  $\RG_{I, J} (X) = X$.

\begin{cor}
Let $(R, \m)$ be a complete local ring with  dualizing complex $D_R$,
and let  $X$ be a bounded $R$-complex with finitely generated homologies. 
\begin{enumerate}
		\item[{(\rm 1)}]\ If $X$ is a $J$-torsion, then $X \cong \RG_{J}(\RG_{\m, J}(X^{\vee})^{\vee})$.
		\item[{(\rm 2)}]\ If $X$ is an $(\m, J)$-torsion, then $X \cong \RG_{\m, J}(\RG_{J}(X^{\vee})^{\vee})$.
\end{enumerate}
\end{cor}

\begin{proof}
$(\rm 1)$ 
Since $X$ is $J$-torsion, it follows from Theorem \ref{p-38} that 
\[ X=\RG_{J}(X) \cong \RG_{J}(\RG_{\m, J} (\RHom (X, D_{R}))^{\vee})[-d], \]
where,  by Theorem \ref{p-53},  
\[ \RG_{\m, J} (\RHom (X, D_{R})) \cong \RG_{\m, J}(\RG_{J} (X)^{\vee})[-d] = \RG_{\m, J}(X^{\vee})[-d].\]
Thus the claim (1) follows. 

$(\rm 2)$ 
Since $X$ is $(\m, J)$-torsion, it holds that 
\[ X=\RG_{\m, J}(X) \cong \RG_{\m, J}(\RG_{J} (\RHom (X, D_{R}))^{\vee})[-d].\]
On the other hand we have from Theorem \ref{p-38} that
\[ \RG_{J} (\RHom (X, D)) \cong \RG_{J}(\RG_{\m, J} (X)^{\vee})[-d] = \RG_{J}(X^{\vee})[-d].\]
\end{proof}

\begin{lem}\label{p-56}
Let $M$ be an $(I, J)$-torsion $R$-module, 
and let  $X$ be a left bounded $R$-complex. 
Then there is an isomorphism
\[ \RHom(M, X) \cong \RHom(M, \RG_{I, J}(X)). \]
\end{lem}

\begin{proof}
Let $E$ be an injective resolution of a complex $X$. 
We will show that $\Hom(M, E^{i}) =\Hom (M,\GIJ(E^{i}))$. 
Let $f \in \Hom(M, E^{i})$ and $x \in M$. 
Since $M$ is $(I, J)$-torsion, there exists an integer $n\geq 0$ such that $I^{n}x\subseteq Jx$. 
Thus we have $I^{n}f(x)\subseteq Jf(x)$, thus  $f(x) \in \GIJ(E^{i})$. 
This shows that $\Im f \subseteq \GIJ(E^{i})$. 
Therefore it holds that
\begin{align*}
\RHom (M,X)
	&\cong \Hom (M, E) \\
	&=\Hom (M, \GIJ (E)) \\
	&\cong \RHom (M, \RG_{I,J}(X)). 
\end{align*}
\end{proof}

\begin{prop}
Let $R$ be a $d$-dimensional Gorenstein complete local ring with maximal ideal $\m$, and $J$ be an ideal of $R$ with $\height~J=r$. 
Then 
\[ \Ass (\HmJ^{d-r}(R)^{\vee} )\cap V(J) = \Min (R/J) = \Ass (H_{J}^{r}(R)). \]
\end{prop}

\begin{proof}
Let $\p \in V(J)$. 
By Theorem \ref{p-38} and Lemma \ref{p-56}, it holds that 
\begin{align*}
 \RHom (R/\p, R)
	&\cong \RHom (R/\p, \RG_{J} (R))\\
	&\cong \RHom (R/\p, \RG_{J} ( \RG_{\m, J} (R)^{\vee} [-d]))\\
	&\cong \RHom (R/\p, \RG_{\m, J} (R)^{\vee}) [-d].
\end{align*}
Thus there is a spectral sequence 
\[ \Ext_{R}^{p}(R/\p, \HmJ^{q}(R)^{\vee}) \Rightarrow \Ext_{R}^{p-q+d}(R/\p,R).\]
Since  $\HmJ^i (R) = 0$ for  $i > d-r$, we see from this spectral sequence that 
$$
\Hom(R/\p, \HmJ^{d-r}(R)^{\vee}) = \Ext_{R}^{r}(R/\p,R).
$$
This shows that  $\p \in \Ass \HmJ ^{d-r} (R)^{\vee}$  if and only if  $\Ext _R^r(R/\p, R)_{\p} \not= 0$  if and only if  $\height~\p = r$. 
The first equality in the proposition follows from this, and the second can be proved in a similar manner. 
\end{proof}

\begin{prop}
Let $R$ be a $d$-dimensional complete Gorenstein local ring with maximal ideal $\m$. 
Then 
\[ \Ass (H_{J}^{d}(R)^{\vee} )\cap \WmJ=\Ass (\GmJ(R)).\]
\end{prop}

\begin{proof}
Let $\p \in \WmJ$. 
By Theorem \ref{p-53} and Lemma \ref{p-56}, it holds that
\begin{align*}
\RHom(R/\p, R) 
	&\cong \RHom (R/\p, \RG_{\m,J}(R))\\
	&\cong \RHom (R/\p, \RG_{\m, J}(\RG_{J}(R)^{\vee})[-d])\\
	&\cong \RHom (R/\p, \RG_{J}(R)^{\vee})[-d].
\end{align*}
Thus there are spectral sequences  
\[ \begin{cases}
\Ext_{R}^{p}(R/\p, \HmJ^{q}(R)) \Rightarrow \Ext_{R}^{p+q}(R/\p, R), ~~\text{and} \vspace{4pt} \\
\Ext_{R}^{p}(R/\p, H_{J}^{q}(R)^{\vee}) \Rightarrow \Ext_{R}^{p-q+d}(R/\p, R).
\end{cases} \]
The first spectral sequence induces 
$\Hom_{R}(R/\p, \GmJ (R)) \cong \Hom_{R} (R/\p, R)$, and the second induces 
$\Hom_{R}(R/\p, H_{J}^{d} (R)^{\vee}) = \Hom_{R}(R/\p, R)$, since  $H_J ^q (R) = 0$  for  $q > d$. 
Thus we have shown 
\[ \Hom_{R}(R/\p, \GmJ (R))_{\p} = 
\Hom_{R}(R/\p, H_{J}^{d} (R)^{\vee})_{\p}, \]
for any  $\p \in \WmJ$. 
Since  $\Ass (\GmJ (R)) \subseteq \WmJ$, the proposition follows.
\end{proof}

{\small 
\begin{center} \textbf{Acknowledgments} \end{center}
The authors express their gratitude to Kazuhide Kakizaki. 
The discussion presented in the first section of this paper is actually based on  his master thesis at Okayama University, which was reported by him in the 24th Symposium on Commutative Ring Theory in Japan. 
As concerns this, see the monograph \cite{KTY1} that is written in Japanese.




\begin{thebibliography}{99}

\bibitem{AB}
{\sc M.\ Auslander} and {\sc M.\ Bridger},
{\it Stable module theory},
Mem.\ Amer.\ Math.\ Soc.\ {\bf 94}, 1969.

\bibitem{Brenner} 
{\sc H.\ Brenner}, 
{\it Grothendieck topologies and ideal closure operations},\\
\texttt{http://arxiv.org/abs/math.AG/0612471}.
%

\bibitem{BS}
{\sc M.\ P.\ Brodmann} and {\sc R.\ Y.\ Sharp},
{\it Local cohomology: an algebraic introduction with geometric applications},
Cambridge University Press, 1998.
%

\bibitem{BH}
{\sc W.\ Bruns} and  {\sc J.\ Herzog},
{\it Cohen-Macaulay rings, revised version},
Cambridge University Press, 1998.
%

\bibitem{Foxby}
{\sc H.-B.\ Foxby}, 
{\it Bounded complexes of flat modules}, 
J.\ Pure Applied Algebra {\bf 15} (1979), 149--172.
%

\bibitem{Greenlees-May}
{\sc J.\ P.\ C.\ Greenlees} and {\sc J.\ P.\ May},
{\it Derived functors of $I$-adic completion and local homology},
J. Algebra {\bf 149} (1992), no. 2, 438--453
		

\bibitem{Grothendieck}
{\sc A.\ Grothendieck},
{\it Local cohomology},
Lecture Notes in Mathematics {\bf 41}, Springer, 1967.
%

\bibitem{Hartshorne}
{\sc R.\ Hartshorne},
{\it Residues and duality},
Lecture Notes in Mathematics {\bf 20}, Springer, 1966.
%

\bibitem{He}
{\sc J.\ Herzog}, 
{\it Komplexe. Aufl$\ddot{\text{o}}$sungen und dualitat in der lokalen Algebra}, Habilitationsschritf, Universit$\ddot{\text{a}}$t Regensburg, 1970.
%

\bibitem{Hochster}
{\sc M.\ Hochster},
{\it The equicharacteristic case of some homological conjectures on local rings}, Bull.\ Amer.\ Math.\ Soc.\ {\bf 80} (1974), 683--686.
%

\bibitem{KTY1}
{\sc K.\ Kakizaki}, {\sc R.\ Takahashi}, and {\sc Y.\ Yoshino},
{\it On vanishing of local cohomology modules and the Lichtenbaum-Hartshorne theorem Part I}, Proceedings of the 24th Symposium on Commutative Ring Theory (2002), 78--83.
%

\bibitem{KTY2}
{\sc K.\ Kakizaki}, {\sc R.\ Takahashi}, and {\sc Y.\ Yoshino},
{\it On vanishing of local cohomology modules and the Lichtenbaum-Hartshorne theorem Part II}, Proceedings of the 24th Symposium on Commutative Ring Theory (2002), 84--90.
%

\bibitem{Kawasaki}
{\sc T.\ Kawasaki},
{\it On Macaulayfication of Noetherian schemes},
Trans.\ Amer.\ Math.\ Soc.\ {\bf 352} (2000), 2517--2552.
%

\bibitem{Matsumura}
{\sc H.\ Matsumura},
{\it Commutative ring theory},
Cambridge University Press, 1986.
%

\bibitem{DNT}
{\sc K.\ Divaani-Aazar}, {\sc R.\ Naghipour}, and {\sc M.\ Tousi},
{\it The Lichtenbaum-Hartshorne theorem for generalized local cohomology and connectedness},
Comm.\ Algebra\ {\bf 30} (2002), no. 8, 3687--3702.

\bibitem{Pskine-Szpiro}
{\sc C.\ Peskine} and {\sc L.\ Szpiro},
{\it Dimension projective finie et cohomologie locale},
Publ.\ Math.\ I.H.E.S. {\bf 42} (1972), 47--119.
%

\bibitem{Roberts}
{\sc P.\ Roberts}, 
{\it Le th$\acute{\text{e}}$or$\grave{\text{e}}$me d'intersection}, 
C.\ R.\ Acad.\ Sc. Paris S$\acute{\text{e}}$r.\ I Math.\ {\bf 304} (1987), 177--180.
%

\bibitem{Schenzel}
{\sc P.\ Schenzel},
{\it Explicit computations around the Lichtenbaum-Hartshorne vanishing theorem}, Manuscripta Math.\ {\bf 78} (1993), no. 1, 57--68.
%




\bigskip
\bigskip

{\sc Ryo Takahashi}\\
{\sc Department of Mathematical Sciences, Faculty of Science,\\
 Shinshu University, 3-1-1 Asahi, Matsumoto, Nagano 390-8621, Japan}\\
{\it E-mail address} : \texttt{takahasi@math.shinshu-u.ac.jp}\\
\\
%
{\sc Yuji Yoshino}\\
{\sc Graduate School of Natural Science and Technology, \\
Okayama University, Okayama 700-8530, Japan}\\
{\it E-mail address} : \texttt{yoshino@math.okayama-u.ac.jp}\\
\\
%
{\sc Takeshi Yoshizawa}\\
{\sc Graduate School of Natural Science and Technology, \\
Okayama University, Okayama 700-8530, Japan}\\
{\it E-mail address} : \texttt{tyoshiza@math.okayama-u.ac.jp}

\end{thebibliography}
\end{document}